\documentclass[a4paper,10pt,leqno]{amsart}
        \title{Almost equivariant maps for td-groups}
       \author{Bartels, A.}
      \address{WWU M\"unster\\
               Mathematisches Institut\\
               Einsteinstr.~62,
               D-48149 M\"unster, Germany}
        \email{a.bartels@wwu.de}
      \urladdr{http://www.math.uni-muenster.de/u/bartelsa} 
       \author{L\"uck, W.}
      \address{Mathematisches Institut der Universit\"at Bonn\\
                Endenicher Allee 60\\
                53115 Bonn, Germany}
        \email{wolfgang.lueck@him.uni-bonn.de}
      \urladdr{http://www.him.uni-bonn.de/lueck}
         \date{April 2024}

   \usepackage[utf8]{inputenc}
  \usepackage{hyperref}
  \usepackage{comment}
  \usepackage{calc}
  \usepackage{enumerate,amssymb,bm}
  \usepackage[arrow,curve,matrix,tips,2cell]{xy}
    \SelectTips{eu}{10} \UseTips
    \UseAllTwocells
  \usepackage{tikz}
  \usetikzlibrary{decorations.pathreplacing}
  \usepackage{color}
  \usepackage{scalerel}
  \usepackage{braket}
  \usepackage{mathtools}
  \usepackage{bbm}
  \usepackage{enumitem}
  \usepackage{float}
  \DeclareMathAlphabet{\matheurm}{U}{eur}{m}{n}

\DeclareMathAlphabet{\matheurm}{U}{eur}{m}{n}

\DeclareMathOperator{\colim}{colim}

\DeclareMathOperator{\id}{id}

\DeclareMathOperator{\SL}{SL}

\newcommand{\fold}[1]{d_{#1\text{-}\mathrm{fol}}}

\newcommand{\CVCYC}{{\calc\hspace{-1pt}\mathrm{vcy}}}

  \newcommand{\IN}{\mathbb{N}}

  \newcommand{\IQ}{\mathbb{Q}}
  \newcommand{\IR}{\mathbb{R}}

  \newcommand{\IZ}{\mathbb{Z}}

  \newcommand{\calc}{\mathcal{C}}

  \newcommand{\calf}{\mathcal{F}}

    \newcommand{\calu}{\mathcal{U}}
  \newcommand{\calv}{\mathcal{V}}
  
  \newcommand{\calw}{\mathcal{W}}

\newcommand{\fol}{{\mathrm{fol}}}

\newcommand{\ballinX}[1]{B_{#1}}

\newcounter{commentcounter}

\theoremstyle{plain}
\newtheorem{theorem}{Theorem}[section]

\newtheorem{lemma}[theorem]{Lemma}

\newtheorem{proposition}[theorem]{Proposition}

\newtheorem{assumption}[theorem]{Assumption}

\newtheorem*{theorem*}{Theorem}
\newtheorem*{theoremA*}{Theorem A}
\newtheorem*{theoremB*}{Theorem B}

\theoremstyle{definition}
\newtheorem{definition}[theorem]{Definition}
\newtheorem{addendum}[theorem]{Addendum}
\newtheorem{example}[theorem]{Example}

\newtheorem{remark}[theorem]{Remark}

\newtheorem*{definition*}{Definition}

\theoremstyle{remark}

\makeatletter\let\c@equation=\c@theorem\makeatother

\theoremstyle{definition}

\newcounter{othercommentcounter}

\newcommand{\e}{{\varepsilon}}
\newcommand{\CAT}{\operatorname{CAT}}
\newcommand{\FS}{\mathit{FS}}

\newcommand{\version}[1]              
{\begin{center} last edited on #1\\
last compiled on \today\\
name of tex-file: \jobname
\end{center}
}

\begin{document}

\maketitle

  \begin{abstract}
    We construct certain maps from buildings associated to td-groups to a space closely
    related to the classifying numerable $G$-space for the family $\CVCYC$ of covirtually
    cyclic subgroups.  These maps are used
    elsewhere to study the K-theory of Hecke
    algebras in the spirit of the Farrell--Jones conjecture.
  \end{abstract}


  \typeout{------------------- Introduction -----------------}

  \section{Introduction}


  The Farrell--Jones conjecture~\cite{Farrell-Jones(1993a)} originated in the surgery theory and has applications to the classification of manifolds, notably it implies (in dimension $\geq 5$) Borel's conjecture on the topological rigidity of aspherical manifolds.
  The conjecture concerns the K- and L-groups of group rings and expresses these in terms of an equivariant homology theory.
  It can be viewed as reducing computations to the case of group rings for virtually cyclic groups.  
  Further information on the conjecture can be found for instance in~\cite{Lueck-Reich(2005)}.  
  
 The main result from this paper are used in~\cite{Bartels-Lueck(2023K-theory_red_p-adic_groups)} to obtain computations for the K-theory of Hecke algebras that are in spirit similar to the Farrell--Jones conjecture.
  It can viewed as extending results from~\cite{Bartels-Lueck(2012CAT(0)flow)} which are a central ingredient to the proof of the Farrell--Jones conjecture for $\CAT(0)$-groups~\cite{Bartels-Lueck(2012annals)} from the setting of discrete groups to td-groups.

  \subsection{Discrete case}\label{subsec:discrete}
  Let $\Gamma$ be a discrete groups.  
  Typically a $\Gamma$-space cannot be both compact and $\Gamma$-CW-complex with
  small isotropy groups.  Compromises between these two properties are central to
  axiomatic results for the Farrell--Jones conjecture, see for
  example~\cite[Sec.~2]{Bartels(2018)}.  For $\CAT(0)$-groups such a compromise was established in~\cite[Main
  Thm]{Bartels-Lueck(2012CAT(0)flow)} and~\cite[Thm~3.4]{Wegner(2012)}.

  For a collection $\calf$ of subgroups and $N \in \IN$ we consider the $n$+$1$-fold join
  $E^N_\calf(\Gamma) := \ast_{i=0}^N (\coprod_{V \in \calf} \Gamma/V)$\footnote{The
    infinite join $\ast_{i=0}^\infty (\coprod_{V \in \calf} \Gamma/V)$ is a model for the
    classifying $\Gamma$-CW-complex for the family consisting of all subgroups of $\Gamma$
    that are subconjugated to one  of the $V \in \calf$,
    compare~\cite[App.~A1]{Baum-Connes-Higson(1994)}.}.  As $E^N_\calf(\Gamma)$ is a
  simplicial complex we can equip it with the $\ell^\infty$-metric $d_E$.  We note that
  this metric is $\Gamma$-invariant.  
  Let $X$ be a $\CAT(0)$-space of finite covering dimension with a cocompact proper isometric $\Gamma$-action.
  We fix a basepoint in 
    $X$ and write $B_R$ for the
  closed ball of radius $R$ around the basepoint.  Let $\pi_R \colon X \to B_R$ be the
  radial projection.  Let $\CVCYC$ be the family of subgroups of $\Gamma$ that admit a map
  to a cyclic group with finite kernel.

  \begin{theorem}\label{thm:X-to-E} There is $N \in \IN$ such that for all finite
    $M \subseteq \Gamma$ and $\epsilon > 0$ there is $\calv \subseteq \CVCYC$ finite such
    that for all $L > 0$ we find $R > 0$ and a (continuous) map
    $f \colon X \to E^N_\calv(G)$ satisfying:
    \begin{enumerate}
    \item\label{thm:X-to-E:Gamma-equiv} for $x \in B_{R+L}, g \in M$ we have
      $d_E(f(gx),gf(x))< \epsilon$;
    \item\label{thm:X-to-E:pi_R} for $x \in B_{R+L}$, $R' \geq R$ we have
      $d_E(f(x),f(\pi_{R'}(x))) < \epsilon$.
    \end{enumerate}
  \end{theorem}

  Theorem~\ref{thm:X-to-E} is the discrete precursor to our main result in the totally
  disconnected case, see Theorem~\ref{thm:X-to-J_intro}.  We discuss in
  Remark~\ref{rem:back-to-discrete} how it is implied by the main result.
  Theorem~\ref{thm:X-to-E} has not been stated before, but it can be viewed as a
  reformulation of the results from~\cite{Bartels-Lueck(2012CAT(0)flow),Wegner(2012)}
  cited above.

  There is a homotopy $\Gamma$-action on $B_R$ where $g \in \Gamma$ acts as
  $x \mapsto \pi_R(gx)$. 
  Roughly speaking~\ref{thm:X-to-E:Gamma-equiv} says that $f$ is
  almost $G$-equivariant and~\ref{thm:X-to-E:pi_R} ensures that the tracks of the
  homotopies of the homotopy action on $B_R$ have small images in $E^N_\calv(G)$.


  \subsection{The setup}\label{subsec:basic_setup}

  Throughout this paper we fix a td-group $G$, i.e., a  locally compact second countable topological Hausdorff group.  
    We also fix an action of $G$ on a locally compact 
  $\CAT(0)$-space $X$ of finite covering  dimension. 
  We assume that the action is by isometries, continuous, cocompact, smooth and proper.  In
  particular, all isotropy groups $G_x$ for the action are compact open in $G$.  The main
  result will also depend on a further more technical assumption.  Informally, the
  assumption is that isotropy groups of geodesics in the space of geodesics with bounded
  period for the $G$-action get only smaller in small neighborhoods, at least on a
  suitable fundamental domain.  Technically this is formulated as
  Assumption~\ref{assum:good-FS_0} using the flow space for $X$.  The main example where
  Assumption~\ref{assum:good-FS_0} is satisfied is the action of a reductive $p$-adic
  group on its associated extended Bruhat-Tits building, see Appendix~\ref{app:Bruhat-Tits}.  This
  is the main example we are interested in.


  \subsection{The space $J_\calf^N(G)$}\label{subsec:J_calf_upper_N} Let $\calf$ be a
  collection of closed subgroups of $G$.  As in the discrete case we can 
  consider for $N \in \IN$ the $N$+$1$-fold join
  \begin{equation*}
    J_\calf^N(G) := {\ast}_{i=0}^N\Big( \coprod_{V \in \calf} G/V \Big).
  \end{equation*}
  For closed subgroups $V, V'$ of $G$ the product $G/V \times G/V'$ can as a $G$-space not
  necessarily be written as a coproduct of orbits.  For this reason $J_\calf^N(G)$ is not
  $G$-CW-complex.  But it is still a numerable $G$-space in the sense
  of~\cite[Def.~2.1]{Lueck(2005s)}\footnote{We note that
    $\colim_{N \to \infty} J_\calf^N(G)$ is a model for the classifying numerable
    $G$-space for the family $\calf$, see~\cite[Def.~2.3]{Lueck(2005s)}
    and~\cite[App.~A1]{Baum-Connes-Higson(1994)}.  We will not need this fact, but it
    motivates the definition.  }.  In contrast to $E_\calf^N(\Gamma)$, there is in general
  no $G$-invariant metric on $J_\calf^N(G)$.  In fact, for a closed (but neither open nor
  compact) subgroup $V$ of $G$ there may be no $G$-invariant metric on the orbit $G/V$
  that generates the topology.  This is a more substantial difficulty to formulating
  Theorem~\ref{thm:X-to-E} for td-groups.  We will explain our solution to this difficulty
  next.


  \subsection{$V$-foliated distance in $G$}\label{subsec:V-fol-distance}
  We can equip $G$
  with a left invariant proper metric $d_G$ that generates the topology of $G$,
  see~\cite[Thm.~4.5]{Haagerup-Przybyszewska(2006)}
  or~\cite[Thm.~1.1]{Abels-Manoussos-Noskov-proper-inv-metric(2011)}.  Let $V$ be a closed
  subgroup of $G$.  As a replacement for the in general not existing $G$-invariant metric
  on $G/V$ we will use the following $V$-foliated distance in $G$.  For $g,g' \in G$,
  $\beta, \eta > 0$ we write
  \begin{equation*}
    \fold{V}(g,g') < (\beta,\eta)
  \end{equation*} 
  if there is $v \in V$ with $d_G(e,v) = d_G(g,gv) < \beta$ and $d_G(gv,g') < \eta$.  Note
  that $\fold{V}$ is left $G$-invariant in the sense that
  \[\fold{V}(g,g') < (\beta,\eta) \Longleftrightarrow \fold{V}(g''g,g''g') < (\beta,\eta)
  \]
  holds for all $g,g',g'' \in G$.

  Two elements $g,g'$ in $G$ satisfy $gV = g'V$ if and only if there exists $\beta > 0$
  such that for every $\eta > 0$ we have $\fold{V}(g,g') < (\beta,\eta)$.  One might be
  tempted to consider
  \begin{equation*}
    d_{G/V}(gV,g'V) := \inf \{ \eta \mid \fold{V}(g,g') < (\beta,\eta) \; \text{for some $\beta > 0$} \}, 
  \end{equation*}
  but this infimum can be $0$ for $gV \neq g'V$.  This happens for example for $V$ the
  subgroup of $\SL_2(\IQ_p)$ consisting of diagonal matrices $g=e$ and $g'$ unipotent.


  \subsection{The space $J_\calf^N(G)^\wedge$}
  Let $\calf$ be a collection of closed subgroups of $G$.  For $N \in \IN$ let
  \begin{equation*}
    J_\calf^N(G)^\wedge := {\ast}_{i=0}^N\Big( \coprod_{V \in \calf} G \Big).
  \end{equation*}
  The projections $G \to G/V$ induce a $G$-equivariant map
  $J_\calf^N(G)^\wedge \to  J_\calf^N(G)$.

  As $\coprod_{V \in \calf} G = G \times \calf$ we can write elements in
  $J_\calf^N(G)^\wedge$ as $[t_0(g_0,V_0),\dots,t_n(g_n,V_n)]$ with $t_i \in [0,1]$,
  $g_i \in G$, $V_i \in \calf$ such that $\sum t_i = 1$.  In this notation we have
  $[t_0(g_0,V_0),\dots,t_n(g_n,V_n)] = [t'_0(g'_0,V'_0),\dots,t'_n(g'_n,V'_n)]$ if and
  only if $t_i=t'_i$ for all $i$ and $(g_i,V_i)=(g'_i,V'_i)$ for all $i$ with
  $t_i \neq 0 \neq t'_i$.

  \subsection{$J$-foliated distance}\label{subsec:foliated-distance-in-bfJ_calf} As
  discussed in Subsection~\ref{subsec:J_calf_upper_N} there is in general no $G$-invariant
  metric on $J_\calf^N(G)$.  As a replacement we work with the following notion of
  foliated distance on $J_\calf^N(G)^\wedge$.  For
  \begin{equation*}
    y = [t_0(g_0,V_0),\ldots,t_N(g_N,V_N)], \; y' = [t'_0(g_0',V'_0),\ldots,t'_N(g'_N,V'_N)] \in J^N_\calf(G)^\wedge
  \end{equation*}
  and $\beta,\eta,\epsilon > 0$ we write
  \begin{equation*}
    \fold{J}(y,y') < (\beta,\eta,\epsilon)
  \end{equation*}
  if $|t_i-t'_i| < \epsilon$ for all $i$ and, in addition, for all $i$ with
  $\max \{ t_i,t'_i \} \geq \epsilon$ we have
  \begin{equation*}
    V_i = V'_i, \quad \text{and} \quad \fold{V_i}(g_i,g'_i) < (\beta,\eta).
  \end{equation*}
  There is a map $J^N_\calf(G)^\wedge \to \ast_{i=0}^N \calv$,
  $[t_0(g_0,V_0),\ldots,t_N(g_N,V_N)] \mapsto [t_0V_0,\ldots,t_NV_N]$.  The first
  requirement implies that the images of $y$ and $y'$ in the join $\ast_{i=0}^N \calv$ are
  of distance $< \epsilon$ with respect to the $\ell^\infty$-metric.  Two points $y$ and
  $y'$ in $J_\calf^N(G)^\wedge$ have the same image under the projection
  $J_\calf^N(G)^\wedge \to J_\calf^N(G)$ if and only if there exists $\beta > 0$ such that
  for any $\epsilon > 0$ and any $\eta > 0$ we have
  $\fold{J}(z,z') < (\beta,\eta,\epsilon)$.


  \subsection{Statement of main result}

  We write now $\CVCYC$ for the family of subgroups $V$ of $G$ which are \emph{covirtually
    cyclic}, i.e., $V$ is compact or there exists an exact sequence of topological groups
  $1\to K \xrightarrow{i} V \to \IZ \xrightarrow{p} 1$, where $i$ is the inclusion of a
  compact open subgroup $K$ of $V$ and $\IZ$ is equipped with the discrete topology.

  As before we fix a base point $b \in X$ and write $\ballinX{R}$ for the closed ball of
  radius $R$ around $b$ in $X$.  We write $\pi_R \colon X \to \ballinX{R}$ for the radial
  projection.

\begin{theorem}[Main Theorem]\label{thm:X-to-J_intro} 
  Suppose that Assumption~\ref{assum:good-FS_0} holds.
  
  \noindent There is $N \in \IN$ such that for all $M \subseteq G$ compact and
  $\epsilon > 0$ there are $\beta > 0$ and $\calv \subseteq \CVCYC$ finite with the
  following property:  For all $\eta > 0$ and all $L > 0$ we find $R > 0$ and a (not
  necessarily continuous) map $f \colon X \to J_\calv^N(G)^\wedge$ satisfying:
  \begin{enumerate}
  \item\label{thm:X-to-J_intro:G-equiv} for $x \in \ballinX{R+L}$, $g \in M$ we have
    $\fold{J}(f(gx),gf(x)) < (\beta,\eta,\epsilon)$;
  \item\label{thm:X-to-J_intro:pi_t} for $x \in \ballinX{R+L}$, $R' \geq R$ we have
    $\fold{J}(f(x),f(\pi_{R'}(x))) < (\beta,\eta,\epsilon)$;
  \item\label{thm:X-to-J_intro:rho-cont} there is $\rho > 0$ such that for all
    $x,x' \in X$ with $d_X(x,x') < \rho$ we have
    $\fold{J}(f(x),f(x')) < (\beta,\eta,\epsilon)$.
  \end{enumerate}
\end{theorem}

\begin{remark}[Quantifiers]
  Using quantifiers the beginning of Theorem~\ref{thm:X-to-J_intro} reads as
  \begin{equation*}
    \exists N\;\forall M,\epsilon\; \exists \beta,\calv \; \forall \eta, L \; \exists R,f \; \text{such that}\ldots 
  \end{equation*}
\end{remark}

\begin{remark}[Failure of continuity]%
 \label{the:}
  The map $f$ appearing in Theorem~\ref{thm:X-to-J_intro} is not necessarily continuous,
  but this should not be viewed as a serious problem;~\ref{thm:X-to-J_intro:rho-cont} is a
  sufficient replacement for continuity.
  
  This issue arise in Proposition~\ref{prop:local}.  We discuss in
  Remark~\ref{rem:failure-of-cont-K(V)} how it might be circumvented with more careful
  bookkeeping.
\end{remark}

\begin{remark}[Strategy of the proof of
  Theorem~\ref{thm:X-to-J_intro}]\label{rem:strategy_of_proof_of_main_theorem}
  The proof of Theorem~\ref{thm:X-to-J_intro} will use the flow space $\FS$
  from~\cite{Bartels-Lueck(2012CAT(0)flow)} that mimics the geodesic flow on
  non-positively curved manifolds.  More precisely the map $f$ from
  Theorem~\ref{thm:X-to-J_intro} will be constructed as a composition
  \begin{equation*}
    X \xrightarrow{f_0} \FS \xrightarrow{f_1} J^N_\calv(G).
  \end{equation*} 
  The map $f_0$ is constructed in Theorem~\ref{thm:X-to-FS}.  This uses the dynamic
  properties of the flow on $\FS$ coming from the $\CAT(0)$-geometry of $X$.  The map
  $f_1$ is constructed in Theorem~\ref{thm:FS-to-J} and uses an adaptation of the long and
  thin covers for flow spaces
  from~\cite{Bartels-Lueck-Reich(2008cover),Kasprowski-Rueping(2017long-thin)} to the case
  of td-groups.
\end{remark}

\begin{remark}[About Assumption~\ref{assum:good-FS_0}]
  We do not know whether our main theorem fails in the absence of
  Assumption~\ref{assum:good-FS_0}.  We do not know whether or whether not
  Assumption~\ref{assum:good-FS_0} is always satisfied.  It is not difficult to check that
  Assumption~\ref{assum:good-FS_0} holds automatically if $G$ is discrete.
  It is not difficult to check that Assumption~\ref{assum:good-FS_0} implies that for
  $\ell > 0$ the collection of all $V_c$ with $\tau_c \leq \ell$ contains up to
  conjugation only finitely many subgroups.  We do not know whether or not the converse
  holds.
\end{remark}

\begin{remark}[Back to the discrete case]\label{rem:back-to-discrete} As any discrete
  group $\Gamma$ is also a td-group we can apply Theorem~\ref{thm:X-to-J_intro} in the
  situation of Subsection~\ref{subsec:discrete}.  Write
  $p \colon J_\calv^N(\Gamma)^\wedge \to J_\calv^N(\Gamma)=E_\calv^N(\Gamma)$ for the
  canonical projection.  As $\Gamma$ is discrete there is $\eta > 0$ such that
  $d_\Gamma(g,g') < \eta$ implies $g=g'$.  For such $\eta$ and all $\beta, \epsilon >0$,
  $y,y' \in J_\calv^N(\Gamma)^\wedge$ we then have
  \begin{equation*}
    \fold{J}(y,y') < (\beta,\eta,\epsilon) \implies d_E(p(y),p(y')) < \epsilon.
  \end{equation*}
  Therefore we can simply compose $f$ from Theorem~\ref{thm:X-to-J_intro} with $p$ to
  obtain Theorem~\ref{thm:X-to-E}.  Note that~\ref{thm:X-to-J_intro:rho-cont} from
  Theorem~\ref{thm:X-to-J_intro} implies that $p \circ f$ is continuous (in fact
  uniformly).
\end{remark}


\subsection{Acknowledgments}\label{subsec:Acknowledgements}
We thank Linus Kramer and Stefan Witzel for helpful comments and discussions.
We are grateful to the referee for further helpful comments and a list of careful corrections.

This work has been supported by the ERC Advanced Grant ``KL2MG-interactions'' (no.
662400) of the second author granted by the European Research Council, by the Deutsche
Forschungsgemeinschaft (DFG, German Research Foundation) \-– Project-ID 427320536 \-- SFB
1442, as well as under Germany's Excellence Strategy \-- GZ 2047/1, Projekt-ID 390685813,
Hausdorff Center for Mathematics at Bonn, and EXC 2044 \-- 390685587, Mathematics
M\"unster: Dynamics \-- Geometry \-- Structure.

The paper is organized as follows:

\tableofcontents


\typeout{------------------- Section 2: The flow space $\FS$ -----------------}

\section{The flow space $\FS$}\label{sec:The_flow_space_FS}


\subsection{Construction of the flow space}

Given a metric space $Z$, denote by $\FS = \FS(Z)$ the associated flow space defined
in~\cite[Section~1]{Bartels-Lueck(2012CAT(0)flow)}.  It consists of all generalized
geodesics.  A \emph{generalized geodesic} is a continuous map $c \colon \IR \to Z$ whose
restriction to some interval\footnote{By an interval we mean a set of the form $[a,b]$,
  $[a,+\infty)$, $(-\infty,b]$ or $(-\infty,+\infty)$} is an isometric embedding and is
locally constant on the complement of this interval.
We do allow that $c$ is constant. 
The metric on $\FS$ is given by
\begin{equation*}
  d_\FS(c,c') := \int_\IR \frac{d_Z(c(t),c'(t))}{2e^{|t|}} \, dt.
\end{equation*}
We recall from~\cite[Prop.~1.7]{Bartels-Lueck(2012CAT(0)flow)} that this metric generates
the topology of uniform convergence on compact subsets.  The flow $\Phi$ on $\FS$ is
defined by
\begin{equation*}
  (\Phi_\tau c)(t) := c(t+\tau).
\end{equation*}
We also write $\FS_\infty = \FS_\infty(Z)$ for the subspace of $\FS$ consisting of all
generalized geodesics that are bi-infinite geodesics, i.e., are nowhere locally constant.


\subsection{Basic facts about the flow space}

For later reference we recall some facts about $\FS$
from~\cite{Bartels-Lueck(2012CAT(0)flow)}.

\begin{lemma}\label{lem:unif-cont-flow}
  Let $(Z,d_Z)$ be a metric space.
  \begin{enumerate}
  \item\label{lem:unif-cont-flow:estimate} The map $\Phi$ is a continuous flow and we have
    for $c,d \in \FS(Z)$ and $\tau,\sigma \in \IR$
    \[
      d_{\FS}\bigl(\Phi_{\tau}(c), \Phi_{\sigma}(d)\bigr) \le e^{|\tau|} \cdot
      d_{\FS}(c,d) + |\sigma - \tau|;
    \]

  \item\label{lem:unif-cont-flow:uniform} For fixed $\alpha$ the map
    $\FS \times [-\alpha,\alpha] \to \FS, \; (c,t) \mapsto \Phi_t(c)$ is uniformly
    continuous.
  \end{enumerate}
\end{lemma}

\begin{proof}
  Assertion~\ref{lem:unif-cont-flow:estimate} is proved
  in~\cite[Lemma~1.3]{Bartels-Lueck(2012CAT(0)flow)} and implies
  assertion~\ref{lem:unif-cont-flow:uniform}.
\end{proof}

\begin{lemma}\label{lem:FS-is-proper}
  Suppose that $Z$ is a proper metric space. Then $\FS(Z)$ is a proper metric space, and
  for any $t \in \IR$ the evaluation map $\FS \to Z$, $c \mapsto c(t)$ is proper and
  uniformly continuous.
\end{lemma}

\begin{proof}
  See~\cite[Prop.~1.9 and Lem.~1.10]{Bartels-Lueck(2012CAT(0)flow)}.
\end{proof}

\begin{lemma}\label{lem:G-on-FS-proper}
  If $G$ acts cocompactly, isometrically, or properly respectively on the proper metric
  space $Z$, then the $G$-action on $\FS$ is cocompact, isometric or proper respectively.
\end{lemma}

\begin{proof} Obviously the $G$-actions on $\FS$ is isometric if $G$ acts isometrically on
  $Z$.  We conclude from Lemma~\ref{lem:properties_of_groups_actions}~%
\ref{lem:properties_of_groups_actions:proper_and_subspaces:surjective_maps_and_proper}
  and~\ref{lem:properties_of_groups_actions:proper_and_subspaces:surjective_maps_and_cocompact}
  and Lemma~\ref{lem:FS-is-proper} that the $G$-action on $\FS$ is proper or cocompact
  respectively if $G$ acts proper or cocompact respectively on $Z$.
\end{proof}


\subsection{Foliated distance in $\FS$}\label{subsec:foliated-distance-FS}

For $c,c' \in \FS$, $\alpha,\delta > 0$ we write
\begin{equation*}
  \fold{\FS}(c,c') < (\alpha,\delta)
\end{equation*}
to mean that there is $t \in [-\alpha,\alpha]$ with $d_\FS(\Phi_t(c),c') < \delta$. We set
\begin{equation*}
  U^\fol_{\alpha,\delta}(c) := \{ c' \in \FS \mid \fold{\FS}(c,c') < (\alpha,\delta) \}.
\end{equation*}

\begin{lemma}[Basics about foliated distance]\label{lem:Basics_about_foliated_distance}\
  \begin{enumerate}
  \item\label{lem:Basics_about_foliated_distance:foliated_versus-distance} For
    $c,c' \in \FS$, we get
    \[
      \fold{\FS}(c,c') < (\alpha,\delta) \; \implies \; d_{\FS}(c',c) < \alpha + \delta;
    \]
  \item\label{Basics_about_foliated_distance:left_invariance} For $c,c' \in \FS$ and
    $g \in G$, we have
    \[
      \fold{\FS}(c,c') < (\alpha,\delta) \Longleftrightarrow \fold{\FS}(gc,gc') <
      (\alpha,\delta).
    \]
  \end{enumerate}
\end{lemma}
\begin{proof}~\ref{lem:Basics_about_foliated_distance:foliated_versus-distance} This
  follows from the triangle inequality, since $\Phi$ has at most unit speed, i.e.,
  $d_{\FS}(\Phi_t(c),c) \le |t|$ for all $c \in \FS$ and $t \in \IR$, by
  Lemma~\ref{lem:unif-cont-flow}~\ref{lem:unif-cont-flow:estimate}.
  \\[1mm]~\ref{Basics_about_foliated_distance:left_invariance} Recall that $d_{\FS}$ left
  $G$-invariant and $\Phi$ is compatible with the $G$-action and hence
  $d_{\FS}(\Phi_t(gc),gc') = d_{\FS}(\Phi_t(c),c')$ for all $c,c' \in \FS$, $g \in G$ and
  $t \in \IR$.
\end{proof}

\begin{lemma}[Symmetry and triangle inequality for the foliated
  distance]\label{lem:sym_plus_triangle-fol} \
  \begin{enumerate}
  \item\label{lem:sym_plus_triangle-fol:sym} For $\alpha > 0$, $\delta > 0$ there is
    $\epsilon > 0$ such that for all $c,c' \in \FS$
    \begin{equation*}
      \fold{\FS}(c,c') < (\alpha,\epsilon)  \;  \implies  \; \fold{\FS}(c',c) <  (\alpha,\delta);    
    \end{equation*}	
  \item\label{lem:sym_plus_triangle-fol:triangle} For $\alpha > 0$, $\delta >0$ there is
    $\epsilon > 0$ such that for all $c,c',c'' \in \FS$
    \begin{equation*}
      \fold{\FS}(c,c'),\fold{\FS}(c',c'') < (\alpha,\epsilon) \;  \implies  \; \fold{\FS}(c,c'') <  (2\alpha,\delta).   
    \end{equation*}  
  \end{enumerate}
\end{lemma}

\begin{proof}~\ref{lem:sym_plus_triangle-fol:sym} Given $\alpha > 0$, $\delta > 0$, we
  conclude from Lemma~\ref{lem:unif-cont-flow}~\ref{lem:unif-cont-flow:estimate}, that
  there is $\epsilon > 0$ such that $d_{\FS}(\Phi_t(c),\Phi_t(c')) < \delta$,  whenever $t \in [-\alpha,\alpha]$ and
  $d_{\FS}(c,c') < \epsilon$.  For~\ref{lem:sym_plus_triangle-fol:sym}, suppose now
  $\fold{\FS}(c,c') < (\alpha,\epsilon)$.  Then there is $t \in [-\alpha,\alpha]$ with
  $d_{\FS}(\Phi_t(c),c') < \epsilon$.  This implies
  $d_{\FS}(c,\Phi_{-t}(c')) = d_{\FS} (\Phi_{-t}(\Phi_t(c)),\Phi_{-t}(c')) < \delta$ and
  therefore $\fold{\FS}(c',c) < (\alpha,\delta)$.
  \\[1mm]~\ref{lem:sym_plus_triangle-fol:triangle} Given $\alpha > 0$, $\delta > 0$, we
  find again $\epsilon > 0$ as in~\ref{lem:sym_plus_triangle-fol:sym}.  If
  $\fold{\FS}(c,c'), \fold{\FS}(c',c'') < (\alpha,\epsilon)$, then there exists
  $t,t' \in [-\alpha,\alpha]$ satisfying
  $d_{\FS}(\Phi_t(c),c'),d_{\FS}(\Phi_{t'}(c),c'') < \epsilon$.  This yields
  $d_{\FS}(\Phi_{t+t'}(c),c'') \le d_{\FS}(\Phi_{t'}(\Phi_t(c)), \Phi_{t'}(c')) +
  d_{\FS}(\Phi_{t'}(c'),c'') < \delta + \epsilon$. Hence
  $\fold{\FS}(c,c'') < (2\alpha,\delta+\epsilon)$.  Using $\delta' := \delta/2$ in place
  of $\delta$ and asking in addition for $\epsilon < \delta/2$
  gives~\ref{lem:sym_plus_triangle-fol:triangle}.
\end{proof}

\subsection{The groups $K_c$ and $V_c$}

For $c \in \FS(X)$ we set~\refstepcounter{theorem}
\begin{enumerate}[label=(\thetheorem\alph*),leftmargin=*]
\item\label{K_c}
  $K_c := \; G_c = \;\{g \in G \mid gc=c \} =\; \{ g \in G \mid gc(t)=c(t) \; \text{for
    all} \; t \in \IR \}$;\vspace{1ex}
\item\label{V_c}
  $V_c \; := \; \{ g \in G \mid \exists t \in \IR \, : \, gc = \Phi_t(c) \}$;\vspace{1ex}
\item\label{tau_c}
  $\tau_c \; \; := \; \inf \{ t > 0 \mid \exists v \in V_c \setminus K_c, \text{with} \;
  \Phi_t(c)=vc \}$.
\end{enumerate}

We use $\inf \emptyset = \infty$.  If $\tau_c < \infty$ then we say that $c$ is
\emph{periodic}.  We have $K_c \subseteq V_c$ as the flow is $G$-equivariant.

\begin{assumption}\label{assum:good-FS_0} There exists $\FS_0 \subseteq \FS$ compact such
  that
  \begin{enumerate}[label=(\thetheorem\alph*),leftmargin=*]
  \item\label{assum:good-FS_0:fund-domain} $G \cdot \FS_0 = \FS$;
  \item\label{assum:good-FS_0:V} for $\ell > 0$ and $c_0 \in \FS_0$ there exists an open
    neighborhood $U$ of $c_0$ in $\FS_0$ such that for all $c \in U$ with
    $\tau_c \leq \ell$ we have $V_c \subseteq V_{c_0}$.
  \end{enumerate}
\end{assumption}

\begin{lemma}\label{lem:about-tau_c} Let $c \in \FS$ be periodic.  Then
  $c \in \FS_\infty$.  Moreover, there is $v \in V_c$ with $vc=\Phi_{\tau_c}(c)$ and any
  such $v$ together with $K_c$ generates $V_c$.
\end{lemma}

\begin{proof}
  Consider $c \not\in \FS_\infty$. Then $c$ is constant on some interval $(-\infty,a_-)$ or
  $(a_+,\infty)$.  This implies that for all $v \in V_c$ there is $t$ with $vc(t) = c(t)$.
  In particular, $v$ fixes an endpoint of the image of $c$ which is a finite geodesic or a geodesic ray.
  Since $v$ also fixes the image of $c$ as a set, it must in fact fix $c$.  
  Hence $K_c =  V_c$.
  If $c$ is periodic we have
  $K_c \subsetneq V_c$ and so we must have $c \in \FS_\infty$.
  
  In particular, $t \mapsto \Phi_t(c)$ is injective.  It follows that for $v \in V_c$
  there is a unique $t_v \in \IR$ with $vc=\Phi_{t_v}(c)$.   
  We obtain a group
  homomorphism $V_c \to \IR$, $v \mapsto t_v$ whose kernel is $K_c$.  Denote the image of
  this homomorphism by $\Gamma_c$.  We claim that $\Gamma_c$ is discrete.
  
  To prove this claim, suppose that there are $(v_n)_{n \in \IN}$ with $t_{v_n} \to 0$ as
  $n \to \infty$.  As $c \in \FS_\infty$, $c \colon \IR \to X$ is injective.  Now for any
  $s \in \IR$ we have $v_n(c(s)) = (v_n c)(s) = (\Phi_{t_n}(c)) = c(s + t_n) \to c(s)$,
   so $c(s)$ is an accumulation point of its $G$-orbit.
  But $G \curvearrowright X$ is smooth and proper, so orbits are discrete.  Thus
  $\Gamma_c$ is discrete.
  
  Thus $\tau_c = \min \Gamma_c \cap \IR_{>0}$ and there is $v \in V_c$ with
  $t_v = \tau_c$.  Moreover, $\Gamma_c$ is infinite cyclic and generated by $\tau_c$.
  Thus any $v \in V_c$ with $t_v = \tau_c$ will together with $K_c$ generate $V_c$.
\end{proof}


\typeout{----- Section 4: Triangle inequalities for $\fold{V}$ and $\fold{J}$ -----------}

\section{Triangle inequalities for $\fold{V}$ and
  $\fold{J}$}\label{sec:Triangle_inequality}

We have already proven a version of the triangle inequality for $\fold{\FS}$ on the flow space $\FS$
in Lemma~\ref{lem:sym_plus_triangle-fol}.  
We will need versions of the triangle inequality for $\fold{V}$ and $\fold{J}$ as well.

%

\begin{lemma}\label{lem:unif-continuity} \
  \begin{enumerate}[
                 label=(\thetheorem\alph*),
                 align=parleft, 
                 leftmargin=*,
                 labelindent=1pt,
                 ] 
  \item\label{lem:unif-continuity:M} Let $M \subseteq G$ be compact. For any
    $\epsilon > 0$ there is $\delta > 0$ such that for all $g,g' \in G$, $v \in M$ we have
    \begin{equation*}
      d_G(g,g') < \delta \implies d_G(gv,g'v) < \epsilon;
    \end{equation*}
  \item\label{lem:unif-continuity:V} Let $\alpha \geq 0$.  Then for any $\epsilon > 0$
    there is $\delta > 0$ such that for any closed subgroup $V$ of $G$ and
    $g,g',g'' \in G$ we have
    \begin{equation*}
      \fold{V}(g,g'),  \fold{V}(g',g'') \leq (\alpha,\delta) \quad \implies \quad  \fold{V}(g,g'') \leq (2\alpha,\epsilon).
    \end{equation*}
  \end{enumerate}
\end{lemma}

\begin{proof}
  Assume the first statement fails for given $M \subseteq G$ compact and $\epsilon > 0$.
  Then there are sequences $g_n,g'_n \in G$, $v_n \in M$ 
  with $d_G(g_n,g'_n) < 1/n$ and
  $d_G(g_nv_n,g'_nv_n) \geq \epsilon$.  Let $x_n := g_n^{-1}g'_n$.  Then
  $d_G(e,x_n) = d_G(g_n,g'_n) < 1/n$ and thus $\lim_{n \to \infty} x_n = e$.  By passing
  to a subsequence, we can arrange $\lim_{n \to \infty} v_n = v$ for some $v \in M$. This
  implies $\lim_{n \to \infty} v_n^{-1}x_nv_n = v^{-1}ev=e$.  Hence
  $\lim_{n \to \infty} d_G(g_nv_n,g'_nv_n) = \lim_{n \to \infty} d_G(e,v_n^{-1}x_nv_n) =
  0$, a contradiction.

  For the second statement, let $\alpha > 0$ and $\epsilon > 0$ be given.  Using the
  compactness of the closed $\alpha$-ball in $G$ we find,
  using~\ref{lem:unif-continuity:M}, $\delta > 0$ such that
  \begin{equation*}
    d_G(g,g') < \delta, d_G(v,e) < \alpha \implies d_G(gv,g'v) < \epsilon/2.
  \end{equation*}
  After decreasing $\delta$ we may assume $\delta < \epsilon / 2$.  Now if
  $\fold{V}(g,g'), \fold{V}(g',g'') \leq (\alpha,\delta)$ then there are $v,v' \in G$ with
  $d_G(gv,g'), d_G(g'v',g'') < \delta$, $d_G(v,e),d_G(v',e) < \alpha$.  Our choice of
  $\delta$ implies $d_G(gvv',g'v') < \epsilon / 2$.  Now
  $d_G(vv',e) \leq d_G(vv',v) + d_G(v,e) \leq 2\alpha$ and hence
  $d_G(gvv',g'') \leq d_G(gvv',g'v') + d_G(g'v',g'') < \epsilon/2 + \delta < \epsilon$.
  Thus $\fold{V}(g,g'') < (2\alpha,\epsilon)$.
\end{proof}

\begin{lemma}\label{lem:triangle-ineq-d_calv} Let $\calv \subseteq \CVCYC$ be finite.  Fix
  $\beta > 0$. Then for all $\eta' > 0$ there is $\eta > 0$ with the property that for
  every $\epsilon > 0$ and every $y,y',y'' $ in $J_\calv^N(G)^\wedge$ we have
  \begin{equation*}
    \fold{J}(y,y'), \fold{J}(y',y'') < (\beta,\eta,\epsilon)
    \quad \implies \quad \fold{J}(y,y'') < (2 \beta ,\eta',2 \epsilon).
  \end{equation*}	
\end{lemma}

\begin{proof}
  We use $\eta' := \delta'$ from~\ref{lem:unif-continuity:V} with $\delta := \eta$.
  Write
  \begin{eqnarray*}
    y & = & [t_0(g_0,V_0),\ldots,t_N(g_N,V_N)];
    \\
    y' & = & [t'_0(g'_0,V'_0),\ldots,t'_N(g'_N,V'_N)];
    \\
    y'' & = & [t''_0(g''_0,V''_0),\ldots,t''_N(g''_N,V''_N)].
  \end{eqnarray*}
  As $|t_i-t'_i|, |t'_i-t''_i| < \epsilon$ for all $i$, we have $|t_i-t''_i| < 2\epsilon$ for all $i$.
  
  Suppose that $\max \{t_i,t''_i\} \geq 2\epsilon$.  As
  $|t_i-t'_i|, |t'_i-t''_i| < \epsilon$, this implies both
  $\max \{t_i,t'_i\} \geq \epsilon$ and $\max \{t'_i,t''_i\} \geq \epsilon$.  Thus
  $V_i = V'_i = V''_i$ and $\fold{V_i}(g_i,g'_i) < (\beta,\eta)$ and
  $\fold{V_i}(g'_i,g''_i) < (\beta,\eta)$.  Now~\ref{lem:unif-continuity:V} gives
  $\fold{V_i}(g_i,g''_i) < (2\beta,\eta')$.
\end{proof}


\typeout{------------------- Section 5: Factorization over the flow space
  -----------------}

\section{Factorization over the flow space}\label{sec:factorization}

In this section we prove our main Theorem~\ref{thm:X-to-J_intro} modulo two results about
the flow space $\FS$.


\subsection{From $X$ to $\FS$}\label{sec:From_X_to_J}

\begin{theorem}\label{thm:X-to-FS}
  For all $M \subseteq G$ compact there is $\alpha > 0$ with the following property.  For
  all $\delta > 0$, $L > 0$ there exists $R > 0$ and a uniformly continuous map
  $f_0 \colon X \to \FS$ such that 
  \begin{enumerate}
  \item\label{thm:X-to-FS:G} for $x \in \ballinX{R+L}$, $g \in M$ we have
    $\fold{\FS}(f_0(gx),gf_0(x)) < (\alpha,\delta)$;
  \item\label{thm:X-to-FS:pi} for $x \in \ballinX{R+L}$, $R' \geq R$ we have
    $\fold{\FS}(f_0(x),f_0(\pi_{R'}(x))) < (\alpha,\delta)$, where $\pi_{R'}$ denotes the radial
    projection onto $\ballinX{R+L}$;
  \end{enumerate}
\end{theorem}

\begin{remark}
  Using quantifiers the beginning of Theorem~\ref{thm:X-to-FS} reads as
  \begin{equation*}
    \forall M\; \exists \alpha\; \forall \delta,L\; \exists R,f_0 \; \text{such that}\ldots
  \end{equation*}
\end{remark}

The proof of Theorem~\ref{thm:X-to-FS} is given at the end of
Section~\ref{sec:proof-of-X-to-FS}.


\subsection{From $\FS$ to $|J^N_\calv(G)|^\wedge$}

\begin{theorem}\label{thm:FS-to-J} 
  Suppose that Assumption~\ref{assum:good-FS_0} holds.
  
  \noindent There is $N \in \IN$ such that for any $\alpha > 0$
  and any $\epsilon > 0$ there are $\beta > 0$ and $\calv \subseteq \CVCYC$ finite such that
  for any $\eta > 0$ there are $\delta > 0$, $f_1 \colon \FS \to J_\calv^N(G)^\wedge$, 
  satisfying the following properties.
  \begin{enumerate}
  \item\label{thm:FS-to-J:alpha-delta-to-beta-eta-eps} For $c,c' \in \FS$ with
    $\fold{\FS}(c,c') < (\alpha,\delta)$ we have
    $\fold{J}(f_1(c),f_1(c')) < (\beta,\eta,\epsilon)$;
  \item\label{thm:FS-to-J:equiv-up-to-beta-eta-eps} For $c \in \FS$, $g \in G$ we have
    $\fold{J}(f_1(gc),gf_1(c)) < (\beta,\eta,\epsilon)$.  
  \end{enumerate}
\end{theorem}

\begin{remark}
  Using quantifiers the beginning of Theorem~\ref{thm:FS-to-J} reads as
  \begin{equation*}
    \exists N \; \forall  \alpha,\epsilon \; \exists \beta, \calv\;
    \forall \eta\; \exists \delta,f_1 \; \text{such that} \ldots
  \end{equation*}
\end{remark}

The proof of Theorem~\ref{thm:FS-to-J} (modulo three results proven later) is given at the
end of Section~\ref{sec:three-properties}.


\subsection{Proof of Main Theorem using Theorems~\ref{thm:X-to-FS}
  and~\ref{thm:FS-to-J}}%
\label{subsec:Proof_of_Theorem_ref(thm:X-to-J_intro_using_Theorems_ref(thm:X-to-FS))_and_ref(thm:FS-to-J)}

We restate our Main Theorem~\ref{thm:X-to-J_intro} from the introduction.

\begin{theorem*}[Main Theorem]
  Suppose that Assumption~\ref{assum:good-FS_0} holds.
  
  \noindent There is $N \in \IN$ such that for all $M \subseteq G$ compact and
  $\epsilon > 0$ there are $\beta > 0$ and $\calv \subseteq \CVCYC$ finite with the
  following property:  For all $\eta > 0$ and all $L > 0$ we find $R > 0$ and a (not
  necessarily continuous) map $f \colon X \to J_\calv^N(G)^\wedge$ satisfying:
  \begin{enumerate}
  \item for $x \in \ballinX{R+L}$, $g \in M$ we have
    $\fold{J}(f(gx),gf(x)) < (\beta,\eta,\epsilon)$;
  \item for $x \in \ballinX{R+L}$, $R' \geq R$ we have
    $\fold{J}(f(x),f(\pi_{R'}(x))) < (\beta,\eta,\epsilon)$;
  \item there is $\rho > 0$ such that for all
    $x,x' \in X$ with $d_X(x,x') < \rho$ we have
    $\fold{J}(f(x),f(x')) < (\beta,\eta,\epsilon)$.
  \end{enumerate}
\end{theorem*}

\begin{proof}[Proof of Main Theorem]
  Let $N$ be the number from Theorem~\ref{thm:FS-to-J}.  Let $M \subseteq G$ be compact
  and $\epsilon > 0$.  Theorem~\ref{thm:X-to-FS} gives us a number $\alpha > 0$.
  Theorem~\ref{thm:FS-to-J} gives us $\beta/2 > 0$ and $\calv \subseteq \CVCYC$ finite.  Let
  $\eta > 0$ be given.  Because of Lemma~\ref{lem:triangle-ineq-d_calv} we can find
  $0 < \eta_0 \le \eta$ with the property that every $y,y',y'' $ in $J_\calf^N(G)^\wedge$
  we have
  \refstepcounter{theorem}
  \begin{enumerate}[label=(\thetheorem\alph*),leftmargin=*]
  \item\label{concl_from_Lemma_ref(lem_triangle-ineq-d_calv)}
    $\fold{J}(y,y'), \fold{J}(y',y'') < (\beta/2,\eta_0,\epsilon/2)
     \implies \fold{J}(y,y'') < (\beta ,\eta, \epsilon)$.
  \end{enumerate}
  Theorem~\ref{thm:FS-to-J} gives us $\delta > 0$ and
  $f_1 \colon \FS \to J^N_\calv(G)^\wedge$ satisfying
  \begin{enumerate}[label=(\thetheorem\alph*),leftmargin=*,resume]
  \item\label{nl:alpha-delta-to-beta-eta-eps} if $c,c' \in \FS$ fullfils
    $\fold{\FS}(c,c') < (\alpha,\delta)$, then  we have
    $\fold{J}(f_1(c),f_1(c')) < (\beta/2,\eta_0,\epsilon/2)$;
  \item\label{nl:G-equiv-up-to-beta-eta-eps} for $c \in \FS$, $g \in G$ we have
    $\fold{J}(f_1(gc),gf_1(c)) < (\beta/2,\eta_0,\epsilon/2)$.
  \end{enumerate}
  Let $L > 0$ be given.  Theorem~\ref{thm:X-to-FS} gives us $R > 0$ and
  $f_0 \colon X \to \FS$ uniformly continuous satisfying
  \begin{enumerate}[label=(\thetheorem\alph*),leftmargin=*,resume]
  \item\label{nl:G-equiv-up-to-alpha-delta} for $x \in \ballinX{R+L}$, $g \in M$
    we have $\fold{\FS}(f_0(gx),gf_0(x)) < (\alpha,\delta)$;
  \item\label{nl:pi_t-alpha-delta-small} for $x \in \ballinX{R+L}$, $R \leq t$ we have
    $\fold{\FS}(f_0(x),f_0(\pi_t(x))) < (\alpha,\delta)$.
  \end{enumerate}
  We set now $f := f_1 \circ f_0$.  It remains to verify the three assertions from
  Theorem~\ref{thm:X-to-J_intro}.  
  \\[1ex]~\ref{thm:X-to-J_intro:G-equiv} Let
  $x \in \ballinX{R+L}$, $g \in M$.  Then $\fold{\FS}(f_0(gx),gf_0(x)) < (\alpha,\delta)$
  by~\ref{nl:G-equiv-up-to-alpha-delta}.  Thus~\ref{nl:alpha-delta-to-beta-eta-eps}
  implies $\fold{J}(f_1(f_0(gx)),f_1(gf_0(x))) < (\beta/2,\eta_0,\epsilon/2)$.  On
  the other hand~\ref{nl:G-equiv-up-to-beta-eta-eps} implies
  $\fold{J}\bigl(f_1(gf_0(x)),gf_1(f_0(x))\bigr) < (\beta/2,\eta_0,\epsilon/2)$.
  Now we conclude
  \[\fold{J}(f(gx),gf(x)) = \fold{J}\bigl(f_1(f_0(gx)),gf_1(f_0(x))\bigr) <
    (\beta,\eta,\epsilon)
  \]
  from~\ref{concl_from_Lemma_ref(lem_triangle-ineq-d_calv)}.
  \\[1ex]~\ref{thm:X-to-J_intro:pi_t} 
  Let $x \in \ballinX{R+L}$, $R' \geq R$.  Then
  $\fold{\FS}(f_0(x),f_0(\pi_{R'}(x))) < (\alpha,\delta)$
  by~\ref{nl:pi_t-alpha-delta-small}.  We conclude
  \begin{multline*}
    \fold{J}(f(x),f(\pi_{R'}(x))) =
    \fold{J}(f_1(f_0(x)),f_1(f_0(\pi_{R'}(x))))
    \\
    < (\beta/2,\eta_0,\epsilon/2) \le (\beta,\eta,\epsilon)
  \end{multline*}
  from~\ref{nl:alpha-delta-to-beta-eta-eps}.  
  \\[1ex]~\ref{thm:X-to-J_intro:rho-cont}
  Since $f_0$ is uniformly continuous there is $\rho > 0$ such that
  $d_\FS(f_0(x),f_0(x')) < \delta$ (and in particular
  $\fold{\FS}(f_0(x),f_0(x')) < (\alpha,\delta)$) for all $x,x' \in X$ with
  $d_X(x,x') < \rho$.  Using~\ref{nl:alpha-delta-to-beta-eta-eps} we obtain
  \[
    \fold{J}(f(x),f'(x)) < (\beta/2,\eta_0,\epsilon/2) \le (\beta,\eta,\epsilon)
  \]
  for all $x,x' \in X$ with $d_X(x,x') < \rho$.
\end{proof}


\typeout{------------------- Section 5: The map to the flow space -----------------}

\section{The map to the flow space}\label{sec:proof-of-X-to-FS}

The proof of Theorem~\ref{thm:X-to-FS}, given in this section, will follow
closely arguments from similar results for actions of discrete
groups in~\cite{Bartels-Lueck(2012CAT(0)flow)}.

For $x,x' \in X$ we write $c_{x,x'} \in \FS$ for the generalized geodesic from $x$ to
$x'$, i.e., for the generalized geodesic characterized by
\begin{eqnarray*}
  c_{x,x'}(t) & = & x \quad t \in (-\infty,0], \\
  c_{x,x'}(t) & = & x \quad t \in [d(x,x'),+\infty). 
\end{eqnarray*}

Recall that we fixed a base point $b$ and write $\ballinX{R}$ for the closed ball of radius $R$
around $b$ in $X$.  Recall also that $\pi_R \colon X \to \ballinX{R}$ denotes the radial
projection.

\begin{lemma}\label{lem:unif-cont-x-to-c}
  The map $X \to \FS$, $x \mapsto c_{b,x}$ is uniformly continuous.
\end{lemma}

\begin{proof}
  The map is continuous by~\cite[II.1.4~(1) on page~160]{Bridson-Haefliger(1999)}
  and~\cite[Proposition~1.7]{Bartels-Lueck(2012CAT(0)flow)}.  As it is equivariant for the
  cocompact actions of $G$, it is also uniformly continuous.  (Alternatively, uniform
  continuity can also be checked directly without using the $G$-action.)
\end{proof}

\begin{lemma}\label{lem:pi_R'-estimate}
  For all $\delta > 0$ there is $\Delta > 0$ such that for all $R',T$ with
  $R' \geq T + \Delta$, $x \in X$ we have
  \begin{equation*}
    d_{\FS} (\Phi_T(c_{b,x}), \Phi_T(c_{b,\pi_{R'}(x)})) < \delta.
  \end{equation*}  	
\end{lemma}

\begin{proof}
  Choose $\Delta > 0 $ such that $\int_{\Delta}^{\infty}\frac{s}{2e^{|s|}}ds < \delta$
  holds.  Let $x \in X$.  For $s+T \leq R'$ we have
  \begin{equation*}
    \big(\Phi_T(c_{b,x})\big)(s) = c_{b,x}(s+T) = c_{b,\pi_{R'}(x)}(s+T) = \big(\Phi_T(c_{b,\pi_{R'}(x)})\big)(s),
  \end{equation*}
  while for $s+T \geq R'$ we have
  \begin{align*}
    d_X\big(\big(\Phi_T(c_{b,x})&\big)(s),  \big(\Phi_T(c_{b,\pi_{R'}(x)})\big)(s)\big)
    \\ = & \; d_X\big(c_{b,x}(s+T),c_{b,\pi_{R'}(x)}(s+T)\big)
                          = s + T - R' = s - (R'-T).
  \end{align*}
  Thus, for $R' \geq T + \Delta$,
  \begin{align*}
    d_{\FS} (\Phi_T(c_{b,x}), & \Phi_T(c_{b,\pi_{R'}(x)})) \\
    & = \int_{-\infty}^{\infty}\frac{d_X\big(\big(\Phi_T(c_{b,x})\big)(s), \big(\Phi_T(c_{b,\pi_{R'}(x)})\big)(s)\big)}{2e^{|s|}}ds \\
    & = \int_{R'-T}^{\infty} \frac{s - (R' - T)}{2e^{|s|}}ds \\
    & \leq  \int_{\Delta}^{\infty}\frac{s}{2e^{|s|}}ds\\
   & \le \delta. \qedhere
  \end{align*} 
\end{proof}

\begin{lemma}\label{lem:CAT0-estimate}
  Let $r'$, $L$, $\alpha > 0$, $r'' > \alpha$.  Let $T := r' +r''$,
  $R := r'' + 2r' + \alpha$.  Let $x,x_1,x_2 \in X$ with $d_X(x_1,x_2) \leq \alpha$,
  $d_X(x_1,x) \leq R+L$.  Set $\tau := d_X(x,x_1)-d_X(x,x_2)$.  Then for all
  $t \in [-r',r']$
  \begin{equation*}
    d_X(c_{x_1,x}(T+\tau+t),c_{x_2,x}(T+t)) \leq \frac{2\alpha(L+2r'+2\alpha)}{r''}.
  \end{equation*}	
\end{lemma}

\begin{proof}
  This is an application of the $\CAT(0)$-condition,
  see~\cite[Lemma~3.6(i)]{Bartels-Lueck(2012CAT(0)flow)}\footnote{Strictly speaking, this
  reference gives
  $d_X(c_{x_1,x}(T+t),c_{x_2,x}(T+t+\tau)) \leq \frac{2\alpha(L+2r'+\alpha)}{r''}$.  The
  statement is not quite symmetric in $x_1$ and $x_2$ as we only know
  $d_X(x,x_2) \leq d_X(x,x_1) + d_X(x_1,x_2) \leq R + L + \alpha$.  Thus in
  applying~\cite[Lemma~3.6(i)]{Bartels-Lueck(2012CAT(0)flow)} we need to use
  $L' := L + \alpha$.  In any event, the exact estimate is not important.}.
\end{proof}

\begin{lemma}\label{lem:estimate-d_fol} For all $\alpha > 0$, $\Delta > 0$, $L > 0$ and
  $\delta > 0$ there are $R > 0$, $0 \leq T \leq R - \Delta$ such that for all
  $x,x_1,x_2 \in X$ with $d_X(x_1,x_2) \leq \alpha$, $d_X(x_1,x) \leq R+L$ we have
  \begin{equation*}
    \fold{\FS}(\Phi_T (c_{x_1,x}),\Phi_T (c_{x_2,x})) < (\alpha,\delta).
  \end{equation*} 
\end{lemma}

\begin{proof}
  Pick $1 > \delta' > 0, r' > \Delta, r'' > \alpha$ such that
  \begin{equation*}
    \int_{-\infty}^{-r'} \frac{2|t|+1}{2e^{|t|}} dt < \frac{\delta}{3}, \quad  
    \int_{-r'}^{r'} \frac{\delta'}{2e^{|t|}} dt < \frac{\delta}{3} \quad \text{and} \quad
    \frac{2\alpha(L+2r'+2\alpha)}{r''} < \delta'.  
  \end{equation*}
  Set $R := 2r' + r'' + \alpha$ and $T := r' + r''$.  Then
  $0 \leq T = R - r' - \alpha \leq R - \Delta$.  Set $\tau := d_X(x,x_1) - d_X(x,x_2)$.
  Then $\tau \in [-\alpha,\alpha]$ as $d_X(x_1,x_2) \leq \alpha$.  By
  Lemma~\ref{lem:CAT0-estimate} we have for all $t \in [-r',r']$
  \begin{equation*}
    d_X (c_{x_1,x}(T+\tau+t),c_{x_2,x}(T+t)) < \delta'.
  \end{equation*}
  For $t \in (-\infty,-r']$ we have
  \begin{align*}
    d_X(c_{x_1,x}(T+\tau+t),c_{x_2,x}
    &(T+t)) \\  \leq
    & \quad d_X(c_{x_1,x}(T+\tau+t),c_{x_1,x}(T+\tau-r'))\\
    & + \quad d_X(c_{x_1,x}(T+\tau-r'),c_{x_2,x}(T-r'))\\& + \quad d_X(c_{x_2,x}(T-r'),c_{x_2,x}(T+t))
    \\
    <
    & \quad |t + r'| + \delta' + |t + r'| \\
    = & \quad 2|t+r'| + \delta' \leq 2|t| + \delta' < 2|t| + 1.	
  \end{align*}
  Similarly, for $t \in [r',\infty)$ we have
  \begin{equation*}
    d_X(c_{x_1,x}(T+t),c_{x_2,x}(T+\tau+t)) < 2|t| + 1.
  \end{equation*}
  We can now estimate
  \begin{align*}
    d_\FS & \big( \Phi_{T+\tau} (c_{x_1,x}), \Phi_{T} (c_{x_2,x}) \big) \\ 
          & = \int_{-\infty}^{\infty} \frac{d_X(c_{x_1,x}(T+\tau+t),c_{x_2,x}(T+t))}{2e^{|t|}}dt \\
          & < \int_{-\infty}^{-r'}\frac{2|t| + 1}{2e^{|t|}}dt + \int_{-r'}^{r'}
            \frac{\delta'}{2e^{|t|}}dt + \int_{r'}^{+\infty} \frac{2|t|+1}{2e^{|t|}}dt \\
          & < \frac{\delta}{3} + \frac{\delta}{3} + \frac{\delta}{3} \quad = \quad \delta.
  \end{align*}
  Therefore $\fold{\FS}(\Phi_T (c_{x_1,x}),\Phi_T (c_{x_2,x})) < (\alpha,\delta)$.
\end{proof}

\begin{proof}[Proof of Theorem~\ref{thm:X-to-FS}]
  Let $M \subseteq G$ be compact.  By compactness of $M \cdot b$ there is $\alpha > 0$
  such that $d_X(b,gb) \leq \alpha$ for all $g \in M$.  Let $\delta > 0$, $L > 0$ be
  given.  Let $\Delta$ be the number from Lemma~\ref{lem:pi_R'-estimate}.  Let $R > 0$ and
  $0 \leq T \leq R - \Delta$ be the numbers from Lemma~\ref{lem:estimate-d_fol}.  Let
  $f_0 \colon X \to \FS$ be the map $x \mapsto \Phi_T(c_{b,x})$.  It is uniformly
  continuous by Lemma~\ref{lem:unif-cont-flow} and~\ref{lem:unif-cont-x-to-c}.  We can now
  verify the two assertions from Theorem~\ref{thm:X-to-FS}.  \\[1ex]~\ref{thm:X-to-FS:G}
  Let $g \in M$ and $x \in \ballinX{R+L}$.  Then $f_0(gx) = \Phi_T(c_{gb,gx})$ and
  $gf_0(x) = \Phi_T(c_{b,gx})$.  As $g \in M$, we have $d_X(b,gb) < \alpha$.  Also
  $d_X(gb,gx) = d_X(b,x) \leq R + L$.  We can therefore apply
  Lemma~\ref{lem:estimate-d_fol} and conclude
  \[
    \fold{\FS}(f_0(gx),gf_0(x)) =
    \fold{\FS}(\Phi_T(c_{gb,gx}),\Phi_T(c_{b,gx}))<(\alpha,\delta).
  \]%
\ref{thm:X-to-FS:pi} Let $x \in \ballinX{R+L}$ and $R' \geq R$.  Then
  $f_0(x) = \Phi_T(c_{b,x})$ and $f_0(\pi_{R'}(x)) = \Phi_T(c_{b,\pi_{R'}(x)})$.  As
  $T \leq R - \Delta$ we have $R' \geq R \geq T + \Delta$.  We can therefore apply
  Lemma~\ref{lem:pi_R'-estimate} and conclude
  $d_\FS\bigl(f_0(x),f_0(\pi_{R'}(x))\bigr)=
  d_\FS(\Phi_T(c_{b,x}),\Phi_T(c_{b,\pi_{R'}(x)})) < \delta$.  In particular we get
  \[\fold{\FS}\bigl(f_0(x),f_0(\pi_{R'}(x))\bigr) < (\alpha,\delta).\]
\end{proof}


\typeout{---------------- Section 6: Three properties of the flow
  space. -----------------}

\section{Three properties of the flow space}\label{sec:three-properties}

Recall from
Subsection~\ref{subsec:Proof_of_Theorem_ref(thm:X-to-J_intro_using_Theorems_ref(thm:X-to-FS))_and_ref(thm:FS-to-J)}
and Section~\ref{sec:proof-of-X-to-FS} that to prove our main
Theorem~\ref{thm:X-to-J_intro} it remains to prove Theorem~\ref{thm:FS-to-J} about maps
$\FS \to J_\calv^N(G)$.  In this section we formulate three propositions about the flow
space and use them to prove Theorem~\ref{thm:FS-to-J}.  These three propositions will be
proved in the forthcoming sections.


\subsection{Long thin covers}

\begin{definition}[$\alpha$-long cover]\label{def:alpha_long_cover}
  A cover $\calu$ of the flow space $\FS$ is said to be \emph{$\alpha$-long} if for any
  $c_0 \in \FS$ there exists $U \in \calu$ such that
  $\Phi_{[-\alpha,\alpha]}(c_0) \subseteq U$.
\end{definition}

Typically such covers are thin in directions transversal to the flow and are often
referred to as long thin covers.

\begin{remark}[Discrete versus totally
  disconnected]\label{rem:discrete_versus_totally_disconnected} For a discrete group
  $\Gamma$ it is possible to construct $G$-equivariant maps from the flow space $\FS$ to
  $G$-simplicial complexes (such as $E_\calv^N(\Gamma)$) via the $\alpha$-long
  $\CVCYC$-covers of $\FS$ from~\cite[Sec.~5]{Bartels-Lueck-Reich(2008cover)}
  or~\cite{Kasprowski-Rueping(2017long-thin)}.  Here a $\CVCYC$-cover consists of open
  subsets $U$ of $\FS$ for which there is $V \in \CVCYC$ such that:
  $gU \cap U \neq \emptyset$ iff $g \in V$ and then $gU=U$.  For a td-group $G$ such $U$
  will not exist; elements of $g \in G \setminus V$ that are close to $V$ will typically
  produce $gU\cap U \neq \emptyset$. The way out is to consider $G\cdot U$. In the
  discrete case there is a map $G\cdot U \to G/V$ and this generalizes to the totally
  disconnected case. In fact we will even construct $U \to G$, but this will typically not
  be equivariant and create some $(\beta,\eta)$-errors. 
  Some equalities from the case of discrete groups, we will here be replaced with  foliated
  bounds for the distance. 
  This applies for instance to the map
  $\FS \to J_{\calv}^{N}(G)^{\wedge}$.
\end{remark}

Since $G$-action commutes with the flow, we get also a flow on $G\backslash \FS$.

\begin{lemma}\label{lem:passing_to_the_orbit_space}
  \begin{enumerate}
  \item\label{lem:passing_to_the_orbit_space:down} Let $\calu$ be a $\alpha$-long cover of
    the flow space $\FS$ such that each element $U \in \calu$ is $G$-invariant.  Then
    $\calu' := \{p(U) \mid U \in \calu\}$ is a $\alpha$-long covering of
    $G\backslash \FS$;

  \item\label{lem:passing_to_the_orbit_space:up} Let $\calv$ be a $\alpha$-long covering
    of $G\backslash \FS$. Then $\calv' = \{p^{-1}(V) \mid V \in \calv$ is a
    $\alpha$-long cover of the flow space $\FS$ such that each element $V \in \calu$ is
    $G$-invariant.
  \end{enumerate}
\end{lemma}

\begin{proof}
  This is obvious from the definitions.
\end{proof}


\subsection{Partitions of unity for long thin covers}

\begin{proposition}[Partition of unity]\label{prop:partition} For all $\alpha > 0$,
  $\e > 0$, $N \in \IN$ there is $\alpha' > 0$ such that the following holds.  Let $\calu$
  be a $\alpha'$-long cover of dimension $\leq N$ by $G$-invariant open subsets of $\FS$.
  Then there exists a partition of unity $\{ t_U \colon \FS \to [0,1] \mid U \in \calu \}$
  subordinate to $\calu$ and $\delta > 0$ such that
  \begin{enumerate}
  \item for $U \in \calu$, $c,c' \in \FS$ with $\fold{\FS}(c,c') < (\alpha,\delta)$ we have
    \begin{equation*}
      |t_U(c) -t_U(c')| < \e;  
    \end{equation*}
  \item the $t_U$ are $G$-invariant.
  \end{enumerate}
\end{proposition}

\begin{proof}
  This follows from Lemma~\ref{lem:passing_to_the_orbit_space} and from
  Proposition~\ref{prop:partition-no-G} applied to $G \backslash \FS$.
\end{proof}

\begin{remark}
  Using quantifiers the beginning of Proposition~\label{ref:partition} reads as
  \begin{equation*}
    \forall \alpha,\epsilon,N\; \exists \alpha'\; \forall \,\calu\; \exists \{t_U\}, \delta\;	\text{such that}\ldots 
  \end{equation*}
\end{remark}


\subsection{Dimension of long thin covers}

\begin{proposition}[Dimension of long thin covers]\label{prop:dimension}
    There is $N \in \IN$ such that for any $\alpha' > 0$ there is $\alpha''$ such that the following is true.
    Let $\calw$ be an $\alpha''$-long cover of $\FS$ by $G$-invariant open subsets.
    Then there exist collections $\calu_0,\ldots,\calu_N$ of open $G$-invariant subsets of $\FS$ such that
    \begin{enumerate}
    \item $\calu := \calu_0 \sqcup \ldots \sqcup \calu_N$ is an $\alpha'$-long cover  of $\FS$,
      in particular $\calu_i \cap \calu_j = \emptyset$ for $i \not= j$;
    	\item for each $i$ the open sets in $\calu_i$ are pairwise disjoint;
    	\item for each $U \in \calu = \calu_0 \sqcup \ldots \sqcup \calu_N$
          there is $W \in \calw$ with $U \subseteq W$. 
    \end{enumerate}
\end{proposition}

\begin{proof}
  This follows from Lemma~\ref{lem:passing_to_the_orbit_space} and
  Proposition~\ref{prop:long-thin-cover-non-equivariant} applied to $G \backslash \FS$.
\end{proof}

\begin{remark}
	Using quantifiers the beginning of Proposition~\ref{prop:dimension}  reads as
	\begin{equation*}
          \exists N\; \forall \alpha'\; \exists \alpha''\; \forall \calw\; \exists \calu_0,\ldots,\calu_N\;
          \text{such that}\ldots 
	\end{equation*}
\end{remark}


\subsection{The local structure of $\FS$}

\begin{proposition}[Local structure]\label{prop:local} 
  Suppose that Assumption~\ref{assum:good-FS_0} holds.
  
  \noindent For all $\alpha'' > 0$ there are
  $\beta > 0$ and $\calv \subseteq \CVCYC$ finite with the following property.  For all
  $\eta > 0$ and all $c_0 \in \FS$ there exist $U \subseteq \FS$ open, $h \colon U \to G$,
  $V \in \calv$ and $\delta'' > 0$ such that
  \begin{enumerate}
  \item\label{prop:local:(1)} for some neighborhood $U_0$ of the orbit $Gc_0$ we have
    $\Phi_{[-\alpha'',\alpha'']}(U_0) \subseteq U$;
  \item\label{prop:local:(2)} $U$ is $G$-invariant;
  \item\label{prop:local:(3)} for $c,c' \in U$ we have 
    \begin{equation*}
      \fold{\FS}(c,c') < (\alpha'',\delta'')  \implies \fold{V}(h(c),h(c')) < (\beta,\eta);
    \end{equation*} 
  \item\label{prop:local:(4)} for $c \in U$, $g \in G$ we have;    
  \begin{equation*}
      \fold{V}(h(gc),gh(c)) < (\beta,\eta)
    \end{equation*} 
  \end{enumerate}
\end{proposition}

Proposition~\ref{prop:local} is proven in Section~\ref{sec:local}. 

\begin{remark}
	Using quantifiers the beginning of Proposition~\ref{prop:local} reads as
	\begin{equation*}
          \forall \alpha''\; \exists \beta,\calv\; \forall \eta, c_0\; \exists U,h,\delta''\;
          \text{such that} \ldots
	\end{equation*}
\end{remark}

\begin{remark}[Failure of continuity]\label{rem:failure-of-cont-K(V)} The failure of
  continuity in Theorem~\ref{thm:X-to-J_intro} comes from the failure of continuity of $h$
  in Proposition~\ref{prop:local}.  The action of $G$ on the flow space $\FS$ is not free
  and so maps $h \colon U \to G$ defined on open subsets of the flow space are necessarily
  a bit pathological.  The action of $G$ on the flow space is proper and it seems possible
  to construct maps $U \to G/K(V)$ that are continuous where $K(V)$ is the maximal compact
  subgroup of $V$.  It should then be possible to obtain continuous maps in
  Theorem~\ref{thm:X-to-J_intro} after one replaces $J_\calv^N(G)^\wedge$ with
  $\ast_{i=0}^n \big( \coprod_{V \in \calv} G/K(V) \big)$.  We will not work this out in
  detail here.
\end{remark}


\subsection{Proof of Theorem~\ref{thm:FS-to-J} using
  Propositions~\ref{prop:partition},~\ref{prop:dimension}  and~\ref{prop:local}}%
\label{subsec:proof-FS-to-J}

\begin{proof}[Proof of Theorem~\ref{thm:FS-to-J}]
  We will use for $N$ the number appearing in Proposition~\ref{prop:dimension}.  Given
  $\alpha > 0$ and $\epsilon > 0$ Proposition~\ref{prop:partition} gives us a number
  $\alpha'$ and Proposition~\ref{prop:dimension} gives us then a number $\alpha'' > 0$.
  We can assume $\alpha'' \geq \alpha$.  Next Proposition~\ref{prop:local} gives us a
  number $\beta > 0$ and $\calv \subseteq \CVCYC$ finite.  Let now $\eta > 0$ be given.
  From Proposition~\ref{prop:local} we get for every $c_0 \in \FS$ an open subset
  $U(c_0) \subseteq \FS$, an open neighborhood $U_0(c_0)$ of the orbit $Gc_0$, a map
  $h(c_0) \colon U(c_0) \to G$, an element $V(c_0) \in \calv$ and $\delta''(c_0) > 0$ such
  that the assertions~\ref{prop:local:(1)},~\ref{prop:local:(2)},~\ref{prop:local:(3)}
  and~\ref{prop:local:(4)} hold. Since $G$ acts cocompactly on $\FS$, we can find a finite
  subset $I \subseteq \FS$ such that $\FS = \bigcup_{c_0 \in I} U_0(c_0)$ holds.  Define
  $\delta'' := \min\{\delta''(c_0) \mid c_0 \in I\}$ and
  $\calw = \{U(c_0) \mid c_0 \in I\}$. Then $\calw$ an $\alpha''$-long cover of $\FS$ by
  $G$-invariant open subsets and comes for every $W \in \calw$ with maps
  $h_W \colon W \to G$ and an element $V_W \in \calv$ satisfying
   \begin{enumerate}[label=(\thetheorem\alph*),leftmargin=*]
  \item\label{nl:fol-to-fol} for $W \in \calw$ and $c,c' \in W$ we have
    \begin{equation*}
      \fold{\FS}(c,c') < (\alpha'',\delta'')  \implies d_{V_W-\fol}(h_W(c),h_W(c')) < (\beta,\eta); 
    \end{equation*} 
  \item\label{nl:G-equiv} for $W \in \calw$, $c \in W$ and $g \in G$ we have
    \begin{equation*}
      d_{V_W-\fol}(h_W(gc),gh_W(c)) < (\beta,\eta).
    \end{equation*} 
  \end{enumerate}
  We apply Proposition~\ref{prop:dimension} to the cover $\calw$ and obtain collections
  $\calu_0,\ldots,\calu_N$ of open $G$-invariant subsets of $\FS$ such that
  \begin{enumerate}[label=(\thetheorem\alph*),leftmargin=*,resume]
  \item $\calu := \calu_0 \sqcup \ldots \sqcup \calu_N$ is an $\alpha'$-long cover of $\FS$;
  \item\label{nl:calu_i-disjoint} for each $i$ the open sets in $\calu_i$ are pairwise
    disjoint;
  \item for each $U \in \calu = \calu_0 \sqcup \ldots \sqcup \calu_N$, there is $W_U \in \calw$
    with $U \subseteq W_U$; we pick such a $W_U$ for each $U$ and set $V_U := V_{W_U}$,
    $h_U := h_{W_U}$.
  \end{enumerate}
  Proposition~\ref{prop:partition} yields a $G$-invariant partition of unity
  $\{ t_U \colon \FS \to [0,1] \mid U \in \calu\}$ subordinate to $\calu$ and
  $\delta > 0$, such that
  \begin{enumerate}[label=(\thetheorem\alph*),leftmargin=*,resume]
  \item\label{nl:t_U-t_U-less-eps} for $U \in \calu$, $c,c' \in \FS$ with
    $\fold{\FS}(c,c') < (\alpha,\delta)$ we have
    \begin{equation*}
      |t_U(c) -t_U(c')| < {\epsilon}.  
    \end{equation*}
  \end{enumerate}
  We can assume $\delta \leq \delta''$.  We now define
  $f_1 \colon \FS \to J^N_\calv(G)^\wedge$ by
  \begin{equation*} 
  c \mapsto
    [t_{U_0}(c)(h_{{U_0}}(c),V_{{U_0}}),\ldots,t_{U_N}(c)(h_{U_N}(c),V_{{U_N}})],
  \end{equation*}
  where for $i=0,\ldots,N$, $U_i \in \calu_i$ is determined by $c \in U_i$.  There is at
  most one such $U_i$ by~\ref{nl:calu_i-disjoint}.  If there is no $U_i \in \calu_i$
  containing $c$, then $t_{U}(c) = 0$ for all $U \in \calu_i$ and the choice of $U_i$ is
  without effect.  It remains to verify the two assertions of Theorem~\ref{thm:FS-to-J}.
  \\[1ex]~\ref{thm:FS-to-J:alpha-delta-to-beta-eta-eps} Let $c,c' \in \FS$ with
  $\fold{\FS}(c,c') < (\alpha,\delta)$.  We write
  \begin{eqnarray*}
    f_1(c) & = & [t_{U_0}(c)(h_{{U_0}}(c),V_{{U_0}}),\ldots,t_{U_N}(c)(h_{U_N}(c'),V_{{U_N}})]; \\
    f_1(c') & = & [t_{U'_0}(c')(h_{{U'_0}}(c'),V_{{U'_0}}),\ldots,t_{U'_N}(c')(h_{U'_N}(c'),V_{{U'_N}})].
  \end{eqnarray*} 
  If $\max\{t_{U_i}(c),t_{U'_i}(c')\} \geq \epsilon$  then,
  by~\ref{nl:t_U-t_U-less-eps}, we necessarily have $U_i = U'_i$ and
  $|t_{U_i}(c) - t_{U'_i}(c')| < \epsilon$.    Moreover, by~\ref{nl:fol-to-fol}
  $d_{V_{U_i}-\fol}(h_{U_i}(c),h_{U'_i}(c')) < (\beta,\eta)$. 
  If $\max\{t_{U_i}(c),t_{U'_i}(c')\} < \epsilon$, then  $|t_{U_i}(c) - t_{U'_i}(c')| < \epsilon$,
  as $t_{U_i}(c), t_{U'_i}(c') \in [0,1]$.
  Thus
  $\fold{J}(f_1(c),f_1(c')) < (\beta,\eta,\epsilon)$.
  \\[1ex]~\ref{thm:FS-to-J:equiv-up-to-beta-eta-eps} 
  Let $c \in \FS$ and $g \in G$.  We
  write
  \begin{eqnarray*}
    f_1(c) & = & [t_{U_0}(c)(h_{{U_0}}(c),V_{{U_0}}),\ldots,t_{U_N}(c)(h_{U_N}(c'),V_{{U_N}})]; \\
    f_1(gc) & = & [t_{U'_0}(gc)(h_{{U'_0}},(gc)V_{{U'_0}}),\ldots,t_{U'_N}(gc)(h_{U'_N}(gc),V_{{U'_N}})].
  \end{eqnarray*}
  Then
  \begin{eqnarray*}
    gf_1(c) & = & [t_{U_0}(c)(gh_{{U_0}}(c),V_{{U_0}}),\ldots,t_{U_N}(c)(gh_{U_N}(c),V_{{U'_N}})].
  \end{eqnarray*} 
  As the $U_i$ and the $t_{U_i}$ are all $G$-invariant we have $U_i = U'_i$,
  $t_{U_i}(c) = t_{U'_i}(gc)$ for all $i$.  Moreover, by~\ref{nl:G-equiv},
  $\fold{V_{U_i}}(h_{U_i}(gc),gh_{U_i}(c)) < (\beta,\eta)$.  In particular,
  $\fold{J}(f_1(gc),gf_1(c))<(\beta,\eta,\epsilon)$.
\end{proof}

In order to prove our main Theorem~\ref{thm:X-to-J_intro}, it suffices to prove
Propositions~\ref{prop:partition},~\ref{prop:dimension} and~\ref{prop:local}. This we will
do in the forthcoming sections.


\typeout{------------------- Section 7:  Partition of unity -----------------}

\section{Partition of unity}\label{sec:partition_of_unity}

In this section we finish the proof Proposition~\ref{prop:partition} by proving
Proposition~\ref{prop:partition-no-G} below.

\begin{proposition}[Partition of unity]\label{prop:partition-no-G} Let
  $Z$ be a compact metric space with a flow $\Phi$.
  For all $\alpha > 0$,
  $\epsilon > 0$, $N \in \IN$ there is $\widehat \alpha > 0$ such that
  the following holds.  Let $\calu$ be a $\widehat\alpha$-long cover
  of $Z$ of dimension $\leq N$ by open subsets.  Then there exists a
  partition of unity $\{ t_U \colon Z \to [0,1] \mid U \in \calu \}$
  subordinate to $\calu$ and $\delta > 0$ such that for $U \in \calu$,
  $z,z' \in Z$ with $\fold{\FS}(z,z') < (\alpha,\delta)$ we have
  \begin{equation*}
    |t_U(z) - t_U(z')| < \e.  
  \end{equation*}	
\end{proposition}

\begin{remark}
  Using quantifiers the beginning of Proposition~\ref{prop:partition-no-G} reads as
  \begin{equation*}
    \forall \alpha, \epsilon, N\; \exists \widehat\alpha\; \forall \, \calu\;
    \exists \{t_U\}, \delta > 0\; \forall \, U, z,z'\;	\text{we have}\ldots 
  \end{equation*}
\end{remark}

\begin{lemma}\label{lem:slow-bump-fct}
  Let $\widehat \alpha > 0$.  
  Let $K \subseteq Z$ be compact and $U$ be an open neighborhood of $K$.
  Then there exists a (continuous) map $f \colon Z \to [0,1]$ satisfying
  \begin{enumerate}
  	\item\label{lem:slow-bump-fct:K} $f|_K \equiv 1$;
  	\item\label{lem:slow-bump-fct:not-K} $f|_{Z \setminus \Phi_{[-\widehat\alpha,\widehat\alpha]}U}  \equiv 0$;
  	\item\label{lem:slow-bump-fct:slow} for $s\in \IR$, $z \in Z$
          we have $| f(\Phi_s(z)) - f(z)| \leq \frac{|s|}{\widehat\alpha}$. 
  \end{enumerate}	
\end{lemma}

\begin{proof}
  As $Z$ is a metric space, there is $\varphi \colon Z \to [0,1]$ with
  $\varphi|_K \equiv 1$ and $\varphi|_{Z \setminus U} \equiv 0$.
  Define $F \colon Z \times \IR \to [0,1]$ by
  \begin{equation*}
    F(z,t) := \begin{cases} (1 - \frac{|t|}{\widehat\alpha} )\varphi(\Phi_t(z))
      & t \in [-\widehat\alpha,\widehat\alpha];
      \\ 0
      & \text{else}. \end{cases} 
  \end{equation*}
 Put
  \begin{equation*}
    f(z) := \sup \{ F(z,t) \mid t \in \IR\} = \max \{ F(z,t) \mid t \in [-\widehat\alpha,\widehat\alpha] \}.
  \end{equation*}
  We verify that $f$ is continuous.  Assume it is not.  Then there exist $\epsilon > 0$,
  $z \in Z$, and a sequence $(z_n)_{n \ge 0}$ in $Z$ such that $z_n \to z$ in $Z$ and
  $|f(z_n) - f(z)| > \epsilon$ holds for all $n \ge 0$.  Choose
  $\tau_n \in [-\widehat\alpha,\widehat\alpha]$ with $f(z_n) = F(z_n,\tau_n)$.  We can
  arrange after passing to subsequences that there exists
  $\tau \in [-\widehat\alpha,\widehat\alpha]$ satisfying $\lim_{n \to \infty} \tau_n = \tau$.  Then
  $f(z_n) = F(z_n,\tau_n) \to F(z,\tau) \leq f(z)$.  Hence there is a natural number $N$
  such that $f(z_n) < f(z) + \epsilon$ holds for $n \ge N$. On the other hand, choose
  $\tau' \in [-\widehat\alpha,\widehat\alpha]$ with $f(z) = F(z,\tau')$.  Then
  $f(z_n) \geq F(z_n,\tau') \to f(z)$ and hence there is a natural number $N'$ such that
  $f(z_n) > f(z) - \epsilon$ for $n \ge N'$. This is a contradiction.

  It remains to check~\ref{lem:slow-bump-fct:K},~\ref{lem:slow-bump-fct:not-K}
  and~\ref{lem:slow-bump-fct:slow}.  \\[1ex]~\ref{lem:slow-bump-fct:K} Let $z \in
  K$. Then $F(z,0) = \varphi(z) = 1$. Thus $f(z) = 1$.
  \\[1ex]~\ref{lem:slow-bump-fct:not-K} Let
  $z \in Z \setminus \Phi_{[-\widehat\alpha,\widehat\alpha]}(U)$. Then
  $\Phi_t(z) \not\in U$ for all $t \in [-\widehat\alpha,\widehat\alpha]$. Thus
  $F(z,t) = 0$ for all $t \in [-\widehat\alpha,\widehat\alpha]$. Therefore $f(z) = 0$.
  \\[1ex]~\ref{lem:slow-bump-fct:slow}  
  Next we show for all $s,t \in \IR$ and $z \in Z$.
  \begin{equation}
    \Big|F(\Phi_s(z),t) - F(z,t)\big| \le  \frac{|s|}{\widehat\alpha}.
   \label{lem:slow-bump-fct:auxiliary_1}
  \end{equation}
  If $s$ and $s+t$ belong to $[-\widehat\alpha,\widehat\alpha]$, this follows from
  \begin{eqnarray*}
  \bigl|F(\Phi_s(z),t) - F(z,t+s)\bigr|
   & = &
          \left|(1-\frac{|t|}{\widehat\alpha})\varphi(\Phi_{t}(\Phi_{s}(z)))
          - (1-\frac{|t+s|}{\widehat\alpha})\varphi(\Phi_{t+s}(z) )\right|
    \\
    & = &
          \left|(1-\frac{|t|}{\widehat\alpha})\varphi(\Phi_{t+s}(z))
          - (1-\frac{|t+s|}{\widehat\alpha})\varphi(\Phi_{t+s}(z) )\right|
    \\
    & = &
    \left|\frac{|t| - |t+s|}{\widehat\alpha}\varphi(\Phi_{t+s}(z))\right|
    \\
    & = &
   \frac{\bigl||t| - |t+s|\bigr|}{\widehat\alpha}\varphi(\Phi_{t+s}(z))
    \\
    & \le  &
             \frac{\bigl||t| - |t+s|\bigr|}{\widehat\alpha}
    \\
    & \le &
     \frac{|s|}{\widehat\alpha}.
  \end{eqnarray*}
  Suppose that $s \in [-\widehat\alpha,\widehat\alpha]$ and
  $(s+t) \notin [-\widehat\alpha,\widehat\alpha]$.  Then $F(z,t+s) = 0$ and
  $\widehat \alpha \le |t+ s| \le|s| + [t]$. Now\eqref {lem:slow-bump-fct:auxiliary_1}
  follows from
  \[
  \bigl|F(\Phi_s(z),t)|
   = 
   (1-\frac{|t|}{\widehat\alpha})\varphi(\Phi_{t}(z))
   \le
     \frac{\widehat \alpha - |t|}{\widehat\alpha}
    \le
    \frac{|s|}{\widehat\alpha}.
  \]
  Suppose that $(s + t) \in [-\widehat\alpha,\widehat\alpha]$ and
  $t \notin [-\widehat\alpha,\widehat\alpha]$.  Then $F(\Phi_s(z),t) = 0$ and
  $\widehat \alpha \le |t| = |(s+t) - s| \le |s + t| + |s|$.  Now~\eqref
  {lem:slow-bump-fct:auxiliary_1} follows from
  \[
  \bigl|F((z),s + t)|
   = 
   (1-\frac{|s+t|}{\widehat\alpha})\varphi(z)
   \le
     \frac{\widehat \alpha - |s + t|}{\widehat\alpha}
    \le
    \frac{|s|}{\widehat\alpha}.
  \]
  If $(s + t) \notin [-\widehat\alpha,\widehat\alpha]$ and
  $t \notin [-\widehat\alpha,\widehat\alpha]$, then $F(\Phi_s(z),t) = F(z,t+s) = 0$ and
  hence~\eqref {lem:slow-bump-fct:auxiliary_1} is true.  This finishes the proof of~\eqref
  {lem:slow-bump-fct:auxiliary_1}.
  
  From the definitions we conclude that there exists $t_0$ and $t_1 \in \IR$ such that for
  all $t \in \IR$ we have
  \begin{eqnarray}
    f(\Phi_s(z))
     & = &  
     F(\Phi_s(z),t_0);
    \label{lem:slow-bump-fct:auxiliary_2}
    \\
    F(\Phi_s(z),t)
     & \le &
    F(\Phi_s(z),t_0);
   \label{lem:slow-bump-fct:auxiliary_3}
    \\
    f(z)
     & = &  
     F(z,t_1);
    \label{lem:slow-bump-fct:auxiliary_4}
    \\
    F(z,t)
     & \le &
    F(z,t_1).
   \label{lem:slow-bump-fct:auxiliary_5}
    \end{eqnarray}
  We estimate
  \[
    f(z)
    \stackrel{\eqref{lem:slow-bump-fct:auxiliary_4}}{=} 
    F(z,t_1)
      \stackrel{\eqref{lem:slow-bump-fct:auxiliary_1}}{\le} 
     F(\Phi_s(z),t_1 - s)  + \frac{|s|}{\widehat\alpha}
        \stackrel{\eqref{lem:slow-bump-fct:auxiliary_3}}{\le } 
    F(\Phi_s(z),t_0)  + \frac{|s|}{\widehat\alpha}
        \stackrel{\eqref{lem:slow-bump-fct:auxiliary_2}}{=} 
    f(\Phi_s(z)) + \frac{|s|}{\widehat\alpha},
 \]                                                     
and
\[
    f(\Phi_s(z))
    \stackrel{\eqref{lem:slow-bump-fct:auxiliary_2}}{=} 
    F(\Phi_s(z),t_0)
      \stackrel{\eqref{lem:slow-bump-fct:auxiliary_1}}{\le} 
     F(z,t_0 + s)  + \frac{|s|}{\widehat\alpha}
        \stackrel{\eqref{lem:slow-bump-fct:auxiliary_5}}{\le } 
    F(z,t_1)  + \frac{|s|}{\widehat\alpha}
        \stackrel{\eqref{lem:slow-bump-fct:auxiliary_4}}{\le } 
    f(z) + \frac{|s|}{\widehat\alpha}.
  \]
  This finishes the proof of Lemma~\ref{lem:slow-bump-fct}.
\end{proof}

      \begin{proof}[Proof of Proposition~\ref{prop:partition-no-G}]
        Let $\alpha,\epsilon$ and $N$ be given.  Pick $\widehat\alpha > 0$ such that
        \[
          \frac{(2N+3)\alpha}{\widehat\alpha} < \epsilon/2.
        \]
        For $U \in \calu$ let
        $U^{-\widehat\alpha} := \{ z \in Z \mid \Phi_{[-\widehat\alpha,\widehat\alpha]}(z)
        \subseteq U \}$.  As $\calu$ is $\widehat\alpha$-long, we can find for every
        $z \in Z$ an open neighborhood $U_0(z)$ and an element $U(z) \in \calu$ such that
        $U_0(z) \subseteq U^{-\widehat\alpha}$.  Since $Z$ is compact, we can find a
        finite subset $I \subseteq Z$ such that $Z = \bigcup_{z \in Z} U_0(z)$.  By
        replacing $\calu$ with $\{U(z) \mid z \in I\}$, we can arrange that both $\calu$
        and $\calu^{-\widehat\alpha} := \{ U^{-\widehat\alpha} \mid U \in \calu\}$ are
        $N$-dimensional finite coverings of $Z$.  As $Z$ is compact we can find a Lebesgue
        number $\ell > 0$ for $\calu$. Define a compact subset $K_U \subseteq U$ by
        $K_U = \{z \in Z \mid d_Z(z,Z \setminus U) \ge \ell \}$.
        Then  $\{K_U \mid U \in \calu\}$ covers $Z$.  For each $U \in \calu$ we now choose $f_U$
        as in Lemma~\ref{lem:slow-bump-fct}, i.e., such that
    \refstepcounter{theorem}    
	\begin{enumerate}[label=(\thetheorem\alph*),leftmargin=*]
  	   \item\label{nl:1onK} $f_U|_{K_U} \equiv 1$;
  	   \item\label{nl:0notonU} $f_U|_{Z \setminus U}  \equiv 0$;
  	   \item\label{nl:slow} for $\tau \in \IR$, $z \in Z$ we have
             $| f_U(z) - f_U(\Phi_\tau(z)) | \leq \frac{|\tau|}{\widehat\alpha}$. 
    \end{enumerate}	
	As $\calu$ is finite, we can normalize the $f_U$ to obtain $t_U \colon Z \to [0,1]$ with
	\begin{equation*}
		t_U (z) := \frac{f_U(z)}{\sum_{U' \in \calu} f_{U'}(z)}.
	\end{equation*}
	Then $\{ t_U \mid U \in \calu \}$ is a partition of unity subordinate to $\calu$.
        For $\tau \in [-\alpha,\alpha]$ and $z \in Z$, we next want to estimate
        $t_U(z) - t_U(\Phi_\tau(z))$.  We abbreviate $x_V := f_V(z)$,
        $x'_V := f_V(\Phi_\tau(u))$ for $V \in \calu$.  By~\ref{nl:1onK} $x_V = 1$,
        $x'_{V'}=1$ for at least one $V,V'$.  In particular
        $\sum_V x_V,\sum_V x'_V \geq 1$.  By~\ref{nl:0notonU} and since the dimension of
        $\calu$ is at most $N$, we have $x_V \neq 0$ for at most $N+1$ different
        $V \in \calv$, and similarly for $x'_V$.  By~\ref{nl:slow}
        $|x_V-x'_V| \leq \frac{\tau}{\widehat\alpha}$.  Using all this we compute for
        $\tau \in [-\alpha,\alpha]$ and $z \in Z$,
        where $V$ and $V'$ run  through $\calu$:
	\begin{eqnarray*}
                \lefteqn{|t_U(z) - t_U(\Phi_\tau(z))|} & & \\
                & = & \Bigg| \frac{x_U}{\sum_V x_V} - \frac{x'_{U}}{\sum_{V'} x'_{V'}} \Bigg| \\
		& =  & \Bigg| \frac{\sum_{V'}x_U x'_{V'}}{\sum_{V,V'} x_V x'_{V'}} - \frac{\sum_{V}x'_U x_{V}}{\sum_{V,V'} x_V x'_{V'}} \Bigg| \\
		& \leq  & \frac{\sum_{V} \big| x_U x'_{V} - x_Ux_V + x_Ux_V - x'_Ux_V \big|}{\sum_{V,V'} x_V x'_{V'}} \\ 
		& \leq  & \frac{\sum_{V}x_U \big|x'_{V} - x_V \big| + \big|x_U - x'_U\big|x_V}{\sum_{V,V'} x_V x'_{V'}} \\
		& \leq & \frac{\sum_V \big|x_V-x'_V\big|}{\sum_{V,V'} x_V x'_{V'}}  + \frac{\big|x_U-x'_U\big|}{\sum_{V'} x'_{V'}} \\
		& \leq & {\sum_V \big|x_V-x'_V\big|}  + {\big|x_U-x'_U\big|} \\
                & \leq & \bigl(\bigl|\{V \in \calu \mid x_V \not= 0\}| + \bigl|\{V \in \calu \mid x'_V \not= 0\}| +1 \bigr)
                               \cdot \max\{\big|x_V-x'_V\big| \mid V,V' \in \calu\}\\
		& \leq & (2(N+1)+1) \frac{\tau}{\widehat\alpha} < \frac{\epsilon}{2}.
	\end{eqnarray*}
	As $Z$ is compact the $t_U$ are uniformly continuous. 
	Since $\calu$ is finite, there is $\delta >0$ such that for $U \in \calu$, $z,z' \in Z$, we have
	\begin{equation*}
		d_Z(z,z') < \delta \; \implies \; |t_U(z)-t_U(z')| < \frac{\epsilon}{2}.
	\end{equation*} 
	Thus 
    \begin{equation*}
     	 \fold{\FS}(z,z') < (\alpha,\delta) \; \implies \;  	   |t_U(z) - t_U(z')| < \epsilon.  \qedhere
    \end{equation*}	  
\end{proof}


\typeout{------------------- Section 8:  Dimension of long thin covers -----------------}

\section{Dimension of long thin covers}\label{sec:Dimension_of_long_thin_covers}

In this section we finish the proof of Proposition~\ref{prop:dimension} by
proving Proposition~\ref{prop:long-thin-cover-non-equivariant}
below. 

\begin{proposition}\label{prop:long-thin-cover-non-equivariant}
  There is $N$ such that for any $\alpha > 0$ there is $\widehat \alpha > 0$ such that the
  following is true.  Let $\calw$ be an $\widehat \alpha$-long cover of $G \backslash \FS$
  by open subsets.  Then there exists collections $\calu_0,\ldots,\calu_N$ of open subsets
  of $G \backslash \FS$ such that
   \begin{enumerate}
   \item $\calu := \calu_0 \sqcup \ldots \sqcup \calu_N$ is an $\alpha$-long cover of
     $G \backslash \FS$;
   \item for each $i$ the open sets in $\calu_i$ are pairwise disjoint;
   \item for each $U \in \calu = \calu_0 \sqcup \ldots \sqcup \calu_N$ there is
     $W \in \calw$ with $U \subseteq W$.
   \end{enumerate}
\end{proposition}

\begin{remark}
  Using quantifiers the beginning of
  Proposition~\ref{prop:long-thin-cover-non-equivariant} reads as
  \begin{equation*}
    \exists N\; \forall \alpha\; \exists \widehat\alpha\; \forall \calw\;
    \exists \calu_0,\ldots,\calu_N\; 	\text{such that} \ldots 
  \end{equation*}
\end{remark}

\begin{proof}[Proof of Proposition~\ref{prop:long-thin-cover-non-equivariant}]
  This follows by combining  Proposition~\ref{prop:long-thin-not-colored}
  and Lemma~\ref{lem:color} below.
\end{proof}

\begin{lemma}\label{lem:G-mod-FS-dim-compact}
	$G \backslash \FS$ is of finite dimension, compact and metrizable.
\end{lemma}

\begin{proof} Recall that $G$ acts on $\FS$ cocompactly, isometrically, and properly and that
  $\FS$ is a proper metric space, see Lemmas~\ref{lem:FS-is-proper}, 
  and Lemma~\ref{lem:G-on-FS-proper}. 
  Hence $G \backslash \FS$ is compact and metrizable.
  A formula for a metric is
  \begin{equation*}
  	d_{G \backslash \FS}([c],[c']) = \min \{ d_\FS(gc,c') \mid g \in G \}.
  \end{equation*}
  
  By~\cite[Prop.~2.9]{Bartels-Lueck(2012CAT(0)flow)} the dimension of
  $\FS \setminus \FS^\IR$ is finite.  As $\FS^\IR \cong X$ is also finite dimensional, the
  sum theorem from dimension theory~\cite[Cor.~1.5.5]{Engelking(1978)} now implies that
  $\FS$ is finite dimensional.  Another result from dimension
  theory~\cite[Thm.~1.12.7]{Engelking(1978)} asserts that for open maps $A \to B$ with
  discrete fibers the dimension of $A$ and $B$ agree\footnote{Strictly speaking the two
    results cited from~\cite{Engelking(1978)} are about inductive dimension, not covering
    dimension. However, it is not difficult to check that $\FS$ is separable, so there is
    no difference between covering dimension and inductive
    dimension~\cite[Thm.~1.7.7]{Engelking(1978)}.}.  The quotient map
  $\FS \to G \backslash \FS$ is open, but as the action of $G$ on $\FS$ is not smooth, the
  fibers of the quotient map are not discrete.  For $R > 0$ let $\FS_{R}$ be the subspace
  of $\FS$ consisting of all generalized geodesics $c \colon \IR \to X$ that are locally
  constant on the complement of $[-R,R]$.  As $\FS_{R}$ is a closed subspace of $\FS$ we
  have $\dim \FS_{R} \leq \dim \FS$.  The action of $G$ on $\FS_{R}$ is smooth, so
  $\FS_{R} \to G \backslash \FS_{R}$ has discrete fibers and
  $\dim G \backslash \FS_{R} = \dim \FS_{R} \leq \dim \FS$ is finite.  There is a
  canonical retract $p_R \colon \FS \to \FS_R$ that sends $c$ to the restriction of $c$ to
  $[-R,R]$, more precisely to the generalized geodesic that agrees with $c$ on $[-R,R]$
  and is locally constant on the complement.  It is not difficult to check that the fibers
  of $p_R$ are of uniformly bounded diameter $\epsilon_R$ with $\epsilon_R \to 0$ as
  $R \to \infty$.  The fibers of the induced map
  $\overline{p}_R \colon G \backslash \FS \to G \backslash \FS_R$ have the same property.
  Write $\calw_{R,\delta}$ for the open cover of $G \backslash \FS_R$ by all balls of
  radius $\delta$.  We can refine $\calw_{R,\delta}$ to an open cover $\calw'_{R,\delta}$
  of dimension $\leq \dim \FS$, i.e., every point of $G \backslash \FS_R$ is contained in
  at most $\dim \FS+1$ sets in $\calw'_{R,\delta}$.  Now let $\calu$ an open cover of
  $G \backslash \FS$.  By compactness $\calu$ has a positive Lebesgue number.  We then
  find $\delta > 0$ and $R > 0$ such that the pull-back
  $\overline{p}_R^*(\calw_{R,\delta})$ refines $\calu$.  It follows that
  $\overline{p}_R^*(\calw'_{R,\delta})$ is a refinement of $\calu$ of dimension
  $\leq \dim \FS$.  Thus $\dim G \backslash \FS \leq \dim \FS < \infty$.
\end{proof}

\begin{proposition}\label{prop:long-thin-not-colored}
   There is $N$ such that for any $\beta > 0$ there is $\widehat \alpha > 0$ such that the following is true.
   Let $\calw$ be an $\widehat \alpha$-long cover of $G \backslash \FS$ by open subsets.
   Then there exists an open cover $\calv$ of  $G \backslash \FS$
   such that 	
   \begin{enumerate} 
   \item\label{prop:long-thin-not-colored:long} $\calv$ is an $\beta$-long;
   \item\label{prop:long-thin-not-colored:dim} $\dim \calv \leq N$;
   \item\label{prop:long-thin-not-colored:thin} for each $U \in \calv$ there is
     $W \in \calw$ with $U \subseteq W$.
   \end{enumerate}
\end{proposition}

\begin{proof}
  The main result from~\cite{Kasprowski-Rueping(2017cov)} almost gives this.  More
  precisely Lemma~\ref{lem:G-mod-FS-dim-compact} allows us to
  apply~\cite[Thm.~1.1]{Kasprowski-Rueping(2017cov)} to $G \backslash \FS$.  Thus there
  exists $N$ only depending on the dimension of $G \backslash \FS$ such that for given
  $\beta > 0$ there exists a cover $\calv$ of $G \backslash \FS$
  satisfying~\ref{prop:long-thin-not-colored:long}
  and~\ref{prop:long-thin-not-colored:dim}.
   
  We will argue below that the construction from~\cite{Kasprowski-Rueping(2017cov)} in
  fact also gives~\ref{prop:long-thin-not-colored:thin}.
  
  We point out that we apply~\cite{Kasprowski-Rueping(2017cov)} to the quotient
  $G \backslash \FS$ which no longer carries a group action.
  In~\cite{Kasprowski-Rueping(2017cov)} an equivariant situation is considered, but we use
  the special case of~\cite{Kasprowski-Rueping(2017cov)} where the group acting on the
  flow space is trivial.
  
  Given $\beta > 0$, let $\hat \alpha := 20 \beta$.  Suppose that $\calw$ is an
  $\hat\alpha$-long cover.  Then for any $c \in G \backslash \FS$ there is $W \in \calw$
  with $\Phi_{[-\hat \alpha,\hat \alpha]}(c) \subseteq W$.  As $W$ is open there is
  $\delta > 0$ such that the $\delta$-neighborhood of
  $\Phi_{[-\hat \alpha,\hat \alpha]}(c)$ is still contained in $W$.  As $G \backslash \FS$
  is compact we can choose $\delta > 0$ uniformly, that is for any
  $c \in G \backslash \FS$ there is $W \in \calw$ containing the $\delta$-neighborhood of
  $\Phi_{[-\hat \alpha,\hat \alpha]}(c)$.
  
  The period of $c \in G \backslash \FS$ is
  $\tau(c) := \inf \{ t > 0 \mid \Phi_{t}(c) = c \}$.  If $\Phi_{t}(c) \neq c$ for all
  $t>0$, then $\tau(c) = \infty$.  Let
  \begin{equation*}
    \big( G \backslash \FS \big)_{> \hat \alpha}
    := \{ c \in G \backslash \FS \mid \tau(c) \in (\hat \alpha,\infty) \}.
  \end{equation*}
  Let $\delta > 0$ be given.  By~\cite[Thm.~5.3]{Kasprowski-Rueping(2017cov)} there exists
  an $\beta$-long cover $\calv_1$ of $\big( G \backslash \FS \big)_{> \hat \alpha}$ of
  dimension $\leq N_1$, where $N_1$ depends only on the dimension of $G \backslash \FS$.
  Moreover, as explained in the last line of the proof
  of~\cite[Thm.~5.3]{Kasprowski-Rueping(2017cov)}, every $V \in \calv_1$ is contained in
  the $\delta$-neighborhood of $\Phi_{[\hat \alpha,\hat \alpha]}(c)$ for some $c$, and
  therefore in some $W \in \calw$.

  Next consider
  \begin{equation*}
    \big( G \backslash \FS \big)_{\leq \hat \alpha}
    := \{ c \in G \backslash \FS \mid \tau(c) \in [0,\hat \alpha] \}.
  \end{equation*}
  By~\cite[Lem.~7.6]{Kasprowski-Rueping(2017cov)} there exists an open cover $\calv_2$ of
  $\big( G \backslash \FS \big)_{\leq \hat \alpha}$ of dimension $\leq N_2$, where $N_2$
  again depends only on the dimension of of $G \backslash \FS$.  Moreover, for each $c$ of
  period $\leq \hat \alpha$ there is $V \in \calu_2$ with $\Phi_\IR(c) \subseteq V$, and
  each $V \in \calv_2$ is contained in the $\delta$-neighborhood of $\Phi_\IR(c)$.  In
  fact, by construction, see the last line of the proof
  of~\cite[Lem.~7.6]{Kasprowski-Rueping(2017cov)}, $c$ can here be chosen to be of period
  $\leq \hat \alpha$.  In particular, $V$ is contained in the $\delta$-neighborhood of
  $\Phi_{[0,\hat \alpha]}(c) = \Phi_\IR(c)$, and therefore in some $W \in \calw$.
    
  We have
  $G \backslash \FS = \big( G \backslash \FS \big)_{\leq \hat \alpha} \sqcup \big( G
  \backslash \FS \big)_{> \hat \alpha}$, where the first set if closed
  by~\cite[Lem.~7.1]{Kasprowski-Rueping(2017cov)} and the second consequently open.  In
  particular, the $V \in \calv_1$ are open in $G \backslash \FS$.  We can
  use~\cite[Lem.~2.7]{Kasprowski-Rueping(2017cov)} to extend the $V \in \calv_2$ to open
  subsets of $G \backslash \FS$ while preserving the properties of $\calv_2$.  The union
  of the two covers is now the needed cover $\calv$.
\end{proof}

\begin{lemma}\label{lem:color}
   Fix a number $N$.
   For any $\alpha > 0$ there is $\beta > 0$ with following property.
   Let $\calv$ be an $\beta$-long cover of $G \backslash \FS$.
   Assume that $\dim \calv \leq N$.
   Then there exists collections $\calu_0,\ldots,\calu_N$ of open subsets of $G \backslash \FS$
   such that 	
   \begin{enumerate}
   \item $\calu := \calu_0 \sqcup \ldots \sqcup \calu_N$ is an $\alpha$-long cover of
     $G \backslash \FS$;
   \item for each $i$ the open sets in $\calu_i$ are pairwise disjoint;
   \item for each $U \in \calu = \calu_0 \sqcup \ldots \sqcup \calu_N$ there is
     $V \in \calv$ with $U \subseteq V$.
   \end{enumerate}
\end{lemma}

The proof is not difficult.  We translate between open covers and maps to simplicial
complexes and for the latter we use barycentric subdivision.

\begin{proof}[Proof of Lemma~\ref{lem:color}]
  The metric on $\FS$ has the property that $d_\FS(\Phi_t(c),c) \leq |t|$ for all
  $c \in \FS$ and $t \in \IR$.  It is not difficult to check that there is metric
  $d_{G \backslash \FS}$ with the same property.  For $\lambda > 0$ we define a metric
  $d_\lambda$ on $G \backslash \FS$ as follows.  For $c,c'\in G \backslash \FS$ set
  \begin{equation*}
    d_{\lambda}(c,c') := \inf \sum_{i=0}^{n} |t_i| + \lambda d_{G \backslash \FS}(\Phi_{t_i} (c_i),c_{i+1})
  \end{equation*}
  where the infimum is taken over all finite sequences $c=c_0,\dots,c_{n+1}=c'$,
  $t_0,\dots,t_n \in \IR$.  Compactness of $G \backslash \FS$ can be used to check that
  for an $\beta$-long cover $\calv$ there is $\lambda > 0$ such that the Lebesgue number
  of $\calv$ with respect to $d_\lambda$ is $\geq \beta$.  Let now $\Lambda$ be the nerve
  of $\calv$, i.e., the simplicial complex that has a vertex $v_V$ for each $V \in \calv$
  and where $v_{V_0},\dots,v_{V_n}$ span a simplex iff
  $V_0 \cap \dots \cap {V_n} \neq \emptyset$.  The dimension of $\calv$ is exactly the
  dimension of $\Lambda$.  We equip $|\Lambda|$ with the $l^1$-metric $d^1$.  There is now
  a map $f \colon G \backslash \FS \to |\Lambda|$ satisfying
  \begin{equation}\label{eq:f-contracts} d_\lambda(c,c') \leq \frac{\beta}{4N} \implies
    d^1(f(c),f(c')) \leq \frac{16N^2}{\beta}d_\lambda(c,c'),
  \end{equation}     
  see~\cite[Prop.~5.3]{Bartels-Lueck-Reich(2008hyper)}.  By its construction the map $f$ has
  the following property: the preimage of the open star of $v_U$ is exactly $U$.  Let now
  $\Lambda'$ be the barycentric subdivision of $\Lambda$.  The vertices of $\Lambda'$
  correspond to the simplices of $\Lambda$.  For $j=0,\dots,N$ let $I_j$ be the set of
  simplices of $\Lambda'$ to $j$-simplices of $\Lambda$ and $\tilde\calu_j$ be the
  collection of open stars around simplices in $I_j$.  Then
  $\tilde \calu := \tilde \calu_0 \sqcup \dots \sqcup \tilde \calu_N$ is an open cover
  $|\Lambda| = |\Lambda'|$ of positive Lebesgue number $L$, where $L$ depends only on the
  dimension of $\Lambda$.  Moreover, for each $j$ the open sets in $\tilde \calu_j$ are
  pairwise disjoint.  We now set $\calu_j := f^*\tilde \calu_j$ and $\calu := f^* \calu$.
  Using estimate~\ref{eq:f-contracts} we see that the Lebesgue number of $\calu$ with
  respect to $d_\Lambda$ is at least $\min \{ \frac{\beta}{4N}, \frac{L\beta}{16N^2} \}$.
  In particular, if we choose $\beta \geq \max \{ 4N\alpha, \frac{16N^2\alpha}{L}\}$, then
  the Lebesgue number of $\calu$ with respect to $d_\lambda$ is at least $\alpha$.  Thus
  $\calu$ is $\alpha$-long.  Finally, each open star for $\Lambda'$ is contained in an
  open star for $\Lambda$.  Thus each $U$ from $\calu$ is contained in some $V$ from
  $\calv$.
\end{proof}

\begin{remark}
  A more careful analysis of the arguments from~\cite{Kasprowski-Rueping(2017cov)} reveals
  that the constructions there also lead to coloured covers.  This leads to a more direct
  proof of Proposition~\ref{prop:long-thin-cover-non-equivariant} and renders
  Lemma~\ref{lem:color} superfluous.
\end{remark}


\typeout{------------------------- Section 9:  Local structure  ----------------------}

\section{Local structure}\label{sec:local}

In section we will prove Proposition~\ref{prop:local}. 


\subsection{Neighborhoods in $\FS$ mapping to $G$}

\begin{lemma}\label{lem:d_FS-vs-d_G-on-compact}
   Let $\FS_0 \subseteq \FS$ be compact.
   For all $\alpha > 0$ there is $\beta > 0$ such that for $g \in G$, $c \in \FS_0$ we have
   \begin{equation*}
   	  d_\FS(gc,c) < \alpha \quad \implies \quad d_G(g,e) < \beta.
   \end{equation*}	
\end{lemma}

\begin{proof}
  Assume this fails for a given $\alpha > 0$.  Then there are sequences $(c_n)_{n \ge 0}$
  in $\FS_0$, and $(g_n)_{n \ge 0}$ in $G$ with $d_{\FS}(g_nc_n,c_n) < \alpha$ but
  \begin{equation*}
    \lim_{n \to \infty} d_G(g_n,e) = \infty.
  \end{equation*}  
  After passing to a subsequence, we can assume $c_0 = \lim_{n \to \infty} c_n$ for some
  $c_0 \in \FS_0$.  We can choose a constant $C > 0$ such that
  $d_{\FS}(c_n,c_0) = d_{\FS}(g_nc_n,g_nc_0) \le C$ for $n \ge 0$. 
  Then the $g_n c_0$ are elements of the closed ball $K$ of radius $\alpha + 2C$ around $c_0$. 
  Since $\FS$ is a proper metric space by Lemma~\ref{lem:FS-is-proper},
  $K$ is compact. The set $\{g \in G \mid g \cdot K \cap K \not= \emptyset\}$ is a compact
  subset of $G$ by
  Lemma~\ref{lem:properties_of_groups_actions}~%
\ref{lem:properties_of_groups_actions:translating_compact_subgroups}
  and contains the sequence $(g_n)_{n \ge 0}$.
  After passing to subsequences again, we can assume $\lim_{n \to \infty} g_n =g$
  for some $g \in G$.   Hence $\lim_{n \to \infty} d_G(g_n,e) = d_G(g,e)$. This contradicts
  $\lim_{n \to \infty} d_G(g_n,e) = \infty$.  
\end{proof}

We have defined $U^{\fol}_{\alpha,\delta}(c_0)$ in
Subsection~\ref{subsec:foliated-distance-FS} and $V_{c_0}$ in~\ref{V_c}.

\begin{proposition}\label{prop:constr-h} Let $\FS_0 \subseteq \FS$ be compact.  For all
  $\alpha > 0$ there is $\beta > 0$ such that the following is true: For all $\eta > 0$,
  $c_0 \in \FS_0$, there are $\delta > 0$ and 
  a (not necessarily continuous) map $h \colon G \cdot U^{\fol}_{\alpha,\delta}(c_0) \to G$ satisfying
   \begin{enumerate}
   	\item\label{prop:constr-h:fol-to-fol} for $c,c' \in G \cdot U^{\fol}_{\alpha,\delta}(c_0)$ we have 
       \begin{equation*}
   	      \fold{\FS}(c,c') < (\alpha,\delta) \quad \implies \quad d_{V_{c_0}-\fol}(h(c),h(c')) < (\beta,\eta);
       \end{equation*}	
    \item\label{prop:constr-h:G} for $g \in G$, $c \in G \cdot U^{\fol}_{\alpha,\delta}(c_0)$ we have
      \begin{equation*}
      	d_{V_{c_0}-\fol}(h(gc),gh(c))< (\beta,\eta).
      \end{equation*}
   \end{enumerate}   
\end{proposition}

\begin{proof}
  Let $\alpha > 0$ be given.  By Lemma~\ref{lem:d_FS-vs-d_G-on-compact} there is
  $\beta > 0$ such that for $g \in G$, $c \in \FS_0$ we have
  \begin{equation}
    d_\FS(gc,c) < 3\alpha \quad \implies \quad d_G(g,e) < \beta.
\label{prop:constr-h:estimate}
  \end{equation}    	
  Next let $\eta > 0$ and $c_0 \in \FS_0$ be given.  For $n \in \IN$ choose
  $h_n \colon G \cdot U_{\alpha,1/n}^{\fol}(c_0) \to G$ such that
  $c \in h_n(c) \cdot U_{\alpha,1/n}^\fol(c_0)$ for all
  $c \in G \cdot U_{\alpha,\delta_n}^{\fol}(c_0)$.  We will show that for all sufficiently
  large $n$ the map $h_n$ satisfies~\ref{prop:constr-h:fol-to-fol}
  and~\ref{prop:constr-h:G}.
  \\[1ex]~\ref{prop:constr-h:fol-to-fol}
  Assume there are
  infinitely many $n$ such that~\ref{prop:constr-h:fol-to-fol} fails.  Then there is
  $I \subseteq \IN$ infinite and for $n \in I$ there are $c_n,c'_n \in \FS$ with
  $c_n \in h_n(c_n) \cdot U_{\alpha,1/n}^\fol(c_0)$,
  $c'_n \in h_n(c'_n) \cdot U_{\alpha,1/n}^\fol(c_0)$,
  $\fold{\FS}(c_n,c_n') < (\alpha,1/n)$ such that
   \begin{equation*}
   	  d_{V_{c_0}-\fol}(h_n(c_n),h_n(c'_n)) < (\beta,\eta)
   \end{equation*} 
   fails.  Lemma~\ref{lem:sym_plus_triangle-fol} implies that we can arrange by possibly
   replacing $I$ by a smaller infinite subset that we have 
   \begin{equation*}
   	  \fold{\FS}(h_n(c_n)c_0,h_n(c'_n)c_0) < (3\alpha,1/n).
   \end{equation*}
   Then, with $a_n := h_n(c'_n)^{-1}h_n(c_n)$,
   \begin{equation*}
   	  \fold{\FS}(a_nc_0,c_0) < (3\alpha,1/n). 
   \end{equation*}
   We conclude from
   Lemma~\ref{lem:Basics_about_foliated_distance}~\ref{lem:Basics_about_foliated_distance:foliated_versus-distance}
   that $a_nc_0$ stays in some closed ball $K$ around $c_0$ with respect to $d_G$. Since $\FS$
   is a proper metric space, $K$ is compact. Since $\FS$ is a proper $G$-space by
   Lemma~\ref{lem:G-on-FS-proper}, we conclude from
   Lemma~\ref{lem:properties_of_groups_actions}~\ref{lem:properties_of_groups_actions:translating_compact_subgroups}
   that $\{g \in G \mid g \cdot K \cap K \not= \emptyset\}$ is a compact subset of $G$ and
   contains the sequence $(a_n)_{n \ge 0}$.  Hence we can arrange by passing to
   subsequences that $\lim_{n \to \infty} a_n = a$ holds in $G$ for some $a \in G$.
   
   We can choose $\tau_n \in [-3\alpha,3\alpha]$ for $n \ge 0$ satisfying
   $d_\FS(a_n\Phi_{\tau_n}(c_0),c_0) < 1/n$.  This implies
   $\lim_{n \to \infty} a_n\Phi_{\tau_n}(c_0) = c_0$.  By passing to subsequences again,
   we can arrange $\lim_{n \to \infty} \tau_n = \tau$ for some
   $\tau \in [-3\alpha,3\alpha]$.  We conclude
   $\lim_{n \to \infty} a_n \Phi_{\tau_n}(c_0) = a\Phi_{\tau}(c_0)$ from
   Lemma~\ref{lem:unif-cont-flow}~\ref{lem:unif-cont-flow:uniform}.  This shows
   $a\Phi_\tau(c_0) = c_0$ and hence $a \in V_{c_0}$.  As the flow is of at most unit
   speed, see Lemma~\ref{lem:unif-cont-flow}~\ref{lem:unif-cont-flow:estimate}, and
   $d_{\FS}$ is left $G$-invariant, we get
   \[
     d_\FS(ac_0,c_0) = d_\FS(ac_0,a\Phi_\tau(c_0)) = d_\FS(c_0,\Phi_\tau(c_0)) \leq |\tau|
     \leq 3\alpha.
   \]
   Hence $d_G(a,e) \leq \beta$ by~\ref{prop:constr-h:estimate}.  Since
   $\lim_{n \to \infty} a_n = a$ and $a \in V_{c_0}$ hold, there exists a natural number
   $N$ such that $d_{V_{c_0}-\fol}(a_n,e) < (\beta,\eta)$ holds for $n \in I$ with
   $n \ge N$. Since $D_G$ and hence $d_{V_{c_0}-\fol}$ are left $G$-invariant, we get
   $d_{V_{c_0}-\fol}(h_n(c_n),h_n(c'_n)) < (\beta,\eta)$ for $n \in I$ with $n \ge N$, a
   contradiction.  \\[1ex]~\ref{prop:constr-h:G} Assume there are infinitely many $n$ such
   that~\ref{prop:constr-h:G} fails.  Then there is $I \subseteq \IN$ infinite and for
   $n \in I$ there are $c_n \in \FS$, $g_n \in G$ with
   $c_n \in h_n(c_n) \cdot U_{\alpha,1/n}^\fol(c_0)$,
   $g_nc_n \in h_n(g_nc_n) \cdot U_{\alpha,1/n}^\fol(c_0)$, such that
   \begin{equation*}
     d_{V_{c_0}-\fol}(h_n(g_nc_n),g_nh_n(c_n)) < (\beta,\eta)
   \end{equation*} 
   fails. Recall that $\fold{\FS}$ is left $G$-invariant, see
   Lemma~\ref{lem:Basics_about_foliated_distance}~\ref{Basics_about_foliated_distance:left_invariance}.
   Hence we get from $c_n \in h_n(c_n) \cdot U_{\alpha,1/n}^\fol(c_0)$ and
   $g_nc_n \in h_n(g_nc_n) \cdot U_{\alpha,1/n}^\fol(c_0)$ that
   $d_{\FS}(c_n,h_n(c_n)c_0) \le (\alpha,1/n)$ and
   $d_{\FS}(c_n,g_n^{-1}h_n(g_nc_n)c_0) \le (\alpha,1/n)$ holds for $n \in I$. Put
   $a_n := h_n(c_n)^{-1}g_n^{-1}h_n(g_nc_n)$.  We conclude from
   Lemma~\ref{lem:sym_plus_triangle-fol} that we can arrange by replacing $I$ by a
   possibly smaller infinite subset that
   \[
     \fold{\FS}(a_nc_0,c_0) < (2\alpha,1/n)
   \]
   holds for $n \in I$.  We conclude from
   Lemma~\ref{lem:unif-cont-flow}~\ref{lem:unif-cont-flow:estimate}, that $a_nc_0$ stays
   in the closed ball $K$ of radius $2\alpha +1$ around $c$.  Since $\FS$ is a proper
   metric space, $K$ is compact. Since $\FS$ is a proper $G$-space by
   Lemma~\ref{lem:G-on-FS-proper}, we conclude from
   Lemma~\ref{lem:properties_of_groups_actions}~\ref{lem:properties_of_groups_actions:translating_compact_subgroups}
   that $g \in G \mid g \cdot K \cap K \not= \emptyset\}$ is a compact subset of $G$ and
   contains the sequence $(a_n)_{n \ge 0}$.  Hence we can arrange by passing to
   subsequences that $\lim_{n \to \infty} a_n = a$ holds in $G$ for some $a \in G$.  There
   are $\tau_n \in [-2\alpha,2\alpha]$ with $d_\FS(a_n\Phi_{\tau_n}(c_0),c_0) < 1/n$.
   This implies $\lim_{n \to \infty} a_n\Phi_{\tau_n}(c_0) = c_0$.  We can arrange by
   passing to subsequences $\lim_{n \to \infty} \tau_n = \tau$ for some
   $\tau \in [-2\alpha,2\alpha]$.  We conclude
   $\lim_{n \to \infty} a_n \Phi_{\tau_n}(c_0) = a\Phi_{\tau}(c_0)$ from
   Lemma~\ref{lem:unif-cont-flow}~\ref{lem:unif-cont-flow:uniform}.  This shows
   $a\Phi_\tau(c_0) = c_0$ and hence $a \in V_{c_0}$.  As the flow is of at most unit
   speed, see Lemma~\ref{lem:unif-cont-flow}~\ref{lem:unif-cont-flow:estimate}, and
   $d_{\FS}$ is left $G$-invariant, we get
   \[
     d_\FS(ac_0,c_0) = d_\FS(ac_0,a\Phi_\tau(c_0)) = d_\FS(c_0,\Phi_\tau(c_0)) \leq |\tau|
     \leq 2\alpha.
   \]
   Hence $d_G(a,e) \leq \beta$ by~\ref{prop:constr-h:estimate}. Since
   $\lim_{n \to \infty} a_n = a$ and $a \in V_{c_0}$ hold, there is a natural number $N$
   such that $d_{V_{c_0}-\fol}(a_n,e) < (\beta,\eta)$ holds for $n \in I$ with $n \ge
   N$. As $d_G$ and hence $d_{V_{c_0}-\fol}$ are left $G$-invariant, we get
   $d_{V_{c_0}-\fol}(h_n(g_nc_n),g_nh_n(c_n)) < (\beta,\eta) $ for $n \in I$ with
   $n \ge N$, a contradiction.
 \end{proof}

 The following addendum to Proposition~\ref{prop:constr-h} strengthens the conclusion in
 the case where the period $\tau_{c_0}$ of $c_0$ defined in~\ref{tau_c} is large relative
 to the given $\alpha$. Recall that we have defined the compact subgroup
 $K_{c_0} \subset G$ to be isotropy group $G_{c_0}$ of $c_0 \in \FS$ in~\ref{K_c}.  The
 difference to Proposition~\ref{prop:constr-h} is that in the two conclusions
 $d_{K_{c_0}-\fol}$ is used, not $d_{V_{c_0}-\fol}$.

\begin{addendum}\label{add:constr-h-non-periodic} Let $\FS_0 \subseteq \FS$ be compact.
  For all $\alpha > 0$ there are $\beta > 0$, $\ell > 0$ such that the following is true.
  For all $\eta > 0$, $c_0 \in \FS_0$ with $\tau_{c_0} > \ell$ there are $\delta > 0$,
  $h \colon G \cdot U^{\fol}_{\alpha,\delta}(c_0) \to G$ satisfying the following.
  \begin{enumerate}
  \item\label{add:constr-h-non-per:fol-to-fol} for
    $c,c' \in G \cdot U^{\fol}_{\alpha,\delta}(c_0)$ we have
    \begin{equation*}
      \fold{\FS}(c,c') < (\alpha,\delta) \quad \implies \quad d_{K_{c_0}-\fol}(h(c),h(c')) < (\beta,\eta);
    \end{equation*}	
  \item\label{add:constr-h-non-per:G} for $g \in G$,
    $c \in G \cdot U^{\fol}_{\alpha,\delta}(c_0)$ we have
    \begin{equation*}
      d_{K_{c_0}-\fol}(h(gc),gh(c))< (\beta,\eta).
    \end{equation*}
  \end{enumerate}
\end{addendum}

\begin{proof}
  We can argue almost exactly as in the proof of Proposition~\ref{prop:constr-h}.  For the
  element $a \in V_{c_0}$ produced in the proof of both conclusions there we also proved
  $d_\FS(ac_0,c_0) \leq 3\alpha$.  Thus if $\tau_{c_0} > \ell := 3\alpha$, then we must
  have $ac_0=c_0$, i.e., $a \in K_{c_0}$ and the conclusions follow for $d_{K_{c_0}-\fol}$
  in place of $d_{V_{c_0}-\fol}$.
\end{proof}

\begin{remark}[The role of $\FS_0$]\label{rem:The_role_of_FS_0}
  The construction of maps $h \colon U \to G$ will depend on a choice of base point for
  the orbit $Gc_0$, namely $c_0$.  The same construction with respect to a different base
  point $g_0c_0$ would also work, but with respect to a different collection of subgroups
  and constant $\beta$.  But the subgroups and constant $\beta$ in
  Proposition~\ref{prop:local} are required to be uniform over all orbits.  Therefore the
  base points for different orbits have to be chosen somewhat consistently; in our
  argument we have done this by using only base points from a fixed compact subset $\FS_0$
  of $\FS$.
\end{remark}


\subsection{Proof of Proposition~\ref{prop:local}}

\begin{lemma}\label{lem:isotropy-in-K}
  Let $c_0 \in \FS$.  Then there exists an open neighborhood $U$ of $c_0$ in $\FS$ and a
  compact open subgroup $K$ of $G$ such that $K_c \subseteq K$ for all $c \in U$.
\end{lemma}

\begin{proof}
  Recall that $G$ acts cocompactly, isometrically, properly, and smoothly on $X$.  There
  is an open neighborhood $W \subseteq X$ of $c_0(0)$ such that $G_x \subseteq G_{c_0(0)}$
  for all $x \in W$, see Lemma~\ref{lem:properties_of_groups_actions}~%
\ref{lem:properties_of_groups_actions:isotropy_groups_in_a_neighborhood}.  Now
  $K := G_{c_0(0)}$ and $U := \{ c \in \FS \mid c(0) \in W \}$ satisfy the assertion.
\end{proof}

The following proof of Proposition~\ref{prop:local} is the only place where we use
Assumption~\ref{assum:good-FS_0}.

\begin{proof}
  Let $\FS_0$ be the compact subset of $\FS$ from Assumption~\ref{assum:good-FS_0}.  Given
  $\alpha > 0$ Proposition~\ref{prop:constr-h} and
  Addendum~\ref{add:constr-h-non-periodic} provide us with numbers $\beta > 0$,
  $\ell > 0$.
  
  Next we use~\ref{assum:good-FS_0:V} and~\ref{lem:isotropy-in-K} to find
  $\calv \subseteq \CVCYC$ finite and a finite cover $\calw$ of $\FS_0$ such that for any
  $W \in \calw$ there are $K_W, V_W \in \CVCYC$ satisfying~\refstepcounter{theorem}
  \begin{enumerate}[label=(\thetheorem\alph*),leftmargin=*]
  \item\label{nl:calm-compact} for all $c \in W$ we have $K_c \subseteq K_W$;
  \item\label{nl:calm-periodic} for all $c \in W$ with $0 < \tau_c \leq \ell$ we have
    $V_c \subseteq V_W$.
  \end{enumerate}
  Let now $\eta > 0$ and $c_0 \in \FS$ be given.  According
  to~\ref{assum:good-FS_0:fund-domain} $\FS_0$ is a fundamental domain for the $G$ action.
  This allows us to choose $g_0 \in G$, $W \in \calw$ such that
  $g_0c_0 \in W \subseteq \FS_0$.  If $\tau_{c_0} = \tau_{g_0c_0} > \ell$, then we set
  $V := K_W$ and note that by~\ref{nl:calm-compact} $K_{g_0c_0} \subseteq V$.  If
  $\tau_{c_0} = \tau_{g_0c_0} \leq \ell$ then we set $V := V_W$ and note that
  by~\ref{nl:calm-periodic} $V_{g_0c_0} \subseteq V$.  Now Proposition~\ref{prop:constr-h}
  and Addendum~\ref{add:constr-h-non-periodic} give us $\delta > 0$ and
  $h \colon G \cdot U^\fol_{\alpha,\delta}(c_0) \to G$ satisfying for
  $c,c' \in G \cdot U^{\fol}_{\alpha,\delta}(c_0)$, $g \in G$
  \begin{eqnarray*}
    \fold{V}(h(c),h(c'))   & < & (\beta,\eta) \quad \text{provided} \; \fold{\FS}(c,c') < (\alpha,\delta);\\
    \fold{V}(h(gc),gh(c)) & < & (\beta,\eta).
  \end{eqnarray*}
  Of course $U := G \cdot U^\fol_{\alpha,\delta}(c_0)$ is $G$-invariant.
\end{proof}

This finishes the proof of our main Theorem~\ref{thm:X-to-J_intro}.


\typeout{------------------- Appendix -----------------}

\appendix


\typeout{--------- Section: The Bruhat-tits building for reductive $p$-adic groups  ----------}

\section{The Bruhat-tits building for reductive $p$-adic groups}\label{app:Bruhat-Tits}

Let $K$ be a non-Archimedian local field, i.e.,  a finite extension of the field of
$p$-adic numbers or the field of formal Laurent series $k((t))$ over a finite field $k$.
Consider an algebraic group $G$ over $K$ whose component of the identity is reductive.
Let $G(K)$ be its group of $K$-points.  We will simply say that $G(K)$ is a reductive
$p$-adic group.  We will need the action of $G(K)$ on the associated (extended) Bruhat-Tits building.
The original reference for the Bruhat-Tits building
is~\cite{Bruhat-Tits(1972),Bruhat-Tits(1984)}.  Summaries of the construction can be found
in~\cite{Tits(1979)} and in~\cite[Sec.~I.1]{Schneider-Stuhler(1997)}.

To set up notation we briefly review aspects of the construction.  Let $A$ be the real
affine space constructed in~\cite[1.2, p.31,32]{Tits(1979)}.  It comes equipped with an
action of a subgroup $N(K)$ of $G(K)$.  The affine space is finite dimensional  and the
action of $N(K)$ is cocompact\footnote{This is not explicitly mentioned~\cite[1.2,
  p.31,32]{Tits(1979)} but follows from the construction.}.  There is also a collection
$\Phi_{\text{af}}$ of affine linear function $\alpha \colon A \to \IR$, these are the
affine roots~\cite[1.6, p.33]{Tits(1979)}.  This set is symmetric, i.e., if
$\alpha \in \Phi_{\text{af}}$ then $-\alpha \in \Phi_{\text{af}}$.  After identifying $A$
with the associated linear space $V$ the affine roots can be described as follows: there
are finitely many linear function $a \colon V \to \IR$ (the roots) and for each $a$ there
is a discrete set $\Gamma_a \subseteq \IR$ such that the affine roots are the maps
$v \mapsto a(v) + l$ where $l \in \Gamma_a$.  (This follows from the discussion
in~\cite[1.6, p.33]{Tits(1979)}, see also~\cite[p.~103]{Schneider-Stuhler(1997)}).
Associated to $\alpha \in \Phi_{\text{af}}$ is the half-apartment
$A_\alpha = \{ x \in A \mid \alpha(x) \geq 0 \}$ and the wall
$\partial A_\alpha = \{ x \in A \mid \alpha(a)=0 \}$.  Chambers of $A$ are the connected
components of the complements of the walls~\footnote{In the quasi-simple case the facets
  are simplices; in the semi-simple case the facets are poly-simplices (i.e., finite
  products of simplices); in general the facets are products of affine spaces with
  poly-simplices~\cite[1.7, p.33]{Tits(1979)}.}.  The facets of the chambers are called
the facets of $A$.  The building $X$ is constructed as a quotient of
$G \times A$~\cite[2.1, p.43]{Tits(1979)}.  The quotient map $G(K) \times A \to X$ is
$G(K)$-equivariant and the $G(K)$-action on $X$ extends the $N(K)$-action on $A$.  The
translates of $A$ under $G(K)$ are the apartments of $X$.  As $N(K)$ acts cocompactly on
$A$ and since $X$ is the union of its apartments the action of $G(K)$ on $X$ is cocompact
as well.  The apartments $gA$ inherit an affine structure and a partition into facets from
$A$; these structures agree on intersections of apartments; any two points (in fact any
two facets) of $X$ are contained in a common apartment~\cite[2.2.1, p.44]{Tits(1979)}.
Given two apartments $A'$ and $A''$ there is $g \in G$ with $gA'=A''$ such that $g$ fixes
$A' \cap A''$ pointwise~\cite[2.2.1]{Tits(1979)}.  This can be used to construct a
$G(K)$-invariant $\CAT(0)$-metric $d_X$ on $X$~\cite[2.3,
p.45]{Tits(1979)}\footnote{In~\cite{Tits(1979)} the terminology of $\CAT(0)$-spaces is not
  used, but the inequality given there is equivalent to the $\CAT(0)$-condition,
  see~\cite[p.~163]{Bridson-Haefliger(1999)}.}.  Apartments are then flat subspace of $X$.
The action of $G(K)$ on $X$ is also proper~\cite[p.45]{Tits(1979)}.  By our assumption on
$K$ its residue field (denoted $\overline{K}$ in~\cite{Tits(1979)}) is finite.  This
assumption is used in some of the following results from~\cite{Tits(1979)}.  The
stabilizer groups of chambers (and therefore of facets) contain the Iwahori
subgroups~\cite[p.54]{Tits(1979)} and these subgroups are open~\cite[p.55]{Tits(1979)}.
In particular, all stabilizer groups for facets are open and the action of $G(K)$ on $X$
is smooth.  The chambers of $X$ can be subdivided to give $X$ the structure of a locally
finite simplicial complex where the action of $G(K)$ is simplicial~\cite[2.3.1,
p.45]{Tits(1979)}.  Altogether $X$ is a finite dimensional $\CAT(0)$-space with a a
proper, continuous, isometric, smooth, cocompact $G(K)$-action.  Assumption~\ref{assum:good-FS_0}
for the action of $G(K)$ on $X$ is verified in
Proposition~\ref{prop:p-adic-groups-have-good-FS_0} below.

\begin{lemma}\label{lem:geodesics-in-apartments} Any generalized geodesic
  $c \colon \IR \to X$ is contained in a translate of $A$.
\end{lemma}

\begin{proof}
  We may assume $c(0) \in A$.  For $n \in \IN$ we find an apartment $A'$ that contains
  $c(\pm n)$.  By the construction of the metric $A'$ will then contain $c([-n,n])$.  Now
  choose $g_n \in G$ such that $g_n A = A'$ and $g_n$ fixes $A \cap A'$.  Then
  $(g_n)_{n \in \IN}$ is a sequence in the compact subgroup of $G(K)$ that fixes $c(0)$
  and has an accumulation point $g$.  As the action of $G(K)$ on $X$ is continuous the
  $g_nA$ must agree with $gA$ on larger and larger neighborhoods of $c(0)$.  It follows
  that the apartment $gA$ contains the image of $c$.
\end{proof}

In the following we write $\FS_\infty$ for the subspace of $\FS$ consisting of all
(bi-infinite) geodesics $c \colon \IR \to X$ and for $Y \subseteq X$ we set
$\FS_\infty(Y) := \FS(Y) \cap \FS_\infty$.

\begin{lemma}\label{lem:distance-to-half-space} Let $c_0 \in \FS_\infty(A)$.  Then there
  is $\epsilon > 0$ such that for all $\alpha \in \Phi_{\text{af}}$, with
  $c_0 \not\in \FS_\infty(A_{\alpha})$ we have
  \begin{equation*}
    d_\FS \big(c_0, \FS_\infty(A_{\alpha}) \big) \geq \epsilon.
  \end{equation*}
\end{lemma}

\begin{proof}
  As there are only finitely many roots it suffices to consider the affine roots
  associated to a fixed root $a$.  The walls associated to these affine roots are then all
  parallel and the half-apartments $A_\alpha$ are linearly ordered by inclusion (because
  $\Gamma_a$ is discrete).  If $c_0$ is not parallel to these walls, then no
  half-apartment $A_\alpha$ contains $c_0$ (or any geodesic parallel to $c_0$) and $c_0$
  intersects all the walls $\partial A_\alpha$ in the same angle.  It is then not
  difficult to bound $d_\FS(c_0, \FS_\infty(A_{\alpha}))$ in terms of this angle.  If $c_0$
  is parallel to the walls $\partial A_\alpha$, then among the $A_\alpha$ not containing
  $c_0$ there is a maximal half-apartment $A_{\alpha_0}$ and we can use
  $\epsilon := d_\FS(c_0,\FS_\infty(A_{\alpha_0}))$.
\end{proof}

\begin{lemma}\label{lem:distance-to-apartements} Let $c_0 \in \FS_\infty(A)$.  Then there
  is $\epsilon > 0$ such that for all $g \in G$ we have
  \begin{equation*}
    d_\FS \big(c_0, \FS_\infty(A \cap gA) \big) \in \{0\} \cup (\epsilon,\infty).
  \end{equation*} 
\end{lemma}

\begin{proof}
  The intersection $gA \cap A$ is a union of facets of $A$ and convex.  It follows that if
  $c_0$ is not contained in $gA \cap A$, then $gA \cap A$ is contained in a half-apartment
  $A_\alpha$ that does not contain $c_0$.  The assertion follows now from
  Lemma~\ref{lem:distance-to-half-space}.
\end{proof}

\begin{lemma}\label{lem:discrete-orbits-in-flats} Let $c_0 \in \FS_\infty(A)$.  Then
  $Gc_0 \cap \FS_\infty(A)$ is discrete.
\end{lemma}

\begin{proof}
  Choose real numbers $t_- < t_+$.  As geodesics in $A$ have unique extensions in $A$ we
  observe the following: if $gc_0 \in \FS_\infty(A)$ and $c_0(t_\pm) = gc_0(t_\pm)$, then
  $c_0 = gc_0$.
    
  Let now $g_n \in G$ with $g_nc_0 \in \FS_\infty(A)$ and
  $g_nc_0 \to c_1 \in \FS_\infty(A)$ as $n \to \infty$.  As the action of $G(K)$ on $X$ is
  smooth all orbits for this action are discrete.  Thus $g_nc_0(t_\pm) = c_1(t_\pm)$ for
  almost all $n$.  The above observation now implies that $g_nc_0$ is eventually constant.
  Thus $Gc_0 \cap \FS(A)$ is discrete, as asserted.
\end{proof}

Recall that an element $c$ in $\FS$ is called periodic if there exists $g \in G$ and
$t \in \IR$ with $t > 0$, and $gc = \Phi_t(c)$.  If $c$ is periodic, then necessarily
$c \in \FS_\infty$.

\begin{lemma}\label{lem:translations-in-nbhd-redone} Let $c_0 \in \FS_\infty(A)$.  
Let $\beta > 0$.  Then there is $\epsilon > 0$ such that the
  following holds.  Let $c \in \FS_\infty(A)$ with $d_\FS(c,c_0) < \epsilon$, $g \in G$,
  $t \in [-\beta,\beta]$ with $gc=\Phi_t c$.  Then $g \in V_{c_0}$.
\end{lemma}

\begin{proof}
  Assume this fails.  Then there are sequences $(c_n)_{n \ge 0}$ in $\FS_\infty(A)$,
  $(t_n)_{n \ge 0}$ in $[-\beta,\beta]$, and $(g_n)_{n \ge 0}$ in $G(K)$ such that
  $\lim_{n \to \infty} c_n = c_0$, $g_n c_n = \Phi_{t_n} c_n$, but $g_n \not\in V_{c_0}$.
  By passing to subsequences we can arrange $\lim_{n \to \infty} t_n = t$ for some
  $t \in [-\beta,\beta]$.  Then we get
  $\lim_{n \to \infty} g_n c_n = \lim_{n \to \infty} \Phi_{t_n}(c_n) = \Phi_{t} c_0$ from
  Lemma~\ref{lem:unif-cont-flow}~\ref{lem:unif-cont-flow:uniform}.  As
  $G \curvearrowright \FS$ is proper, see Lemma~\ref{lem:G-on-FS-proper}, the $g_n$ vary
  over a relatively compact set.  Thus we can pass to a further subsequence and assume
  that $\lim_{n \to \infty} g_n = g$ for some $g \in G$.  Then
  $\lim_{n \to \infty} g_nc_0 = gc_0$.  As $G \curvearrowright \FS$ is isometric we also
  have $\lim_{n \to \infty} g_nc_n = gc_0$.  Thus $gc_0 = \Phi_t(c_0)$.  We have
  $g_nc_n = \Phi_{t_n} c_n \in \FS(A)$.  Thus $c_n \in \FS_\infty((g_n)^{-1}A)$.
  Lemma~\ref{lem:distance-to-apartements} implies that $c_0 \in \FS_\infty((g_n)^{-1}A)$
  for almost all $n$.  Thus $g_nc_0 \in \FS(A)$ for almost all $n$.  Now
  Lemma~\ref{lem:discrete-orbits-in-flats} implies that $g_n c_0 = g c_0$ for almost all
  $n$.  Thus $g_n c_0 = g c_0 = \Phi_t(c_0)$ for almost all $n$, contradicting
  $g_n \not\in V_{c_0}$.
\end{proof}

\begin{lemma}\label{lem:isotropy-decreases-V}
  Let $c_0 \in \FS_\infty(A)$.  
  Let $\ell > 0$.  Then
  there is $\epsilon > 0$ such that for all $c \in \FS_\infty(A)$ with
  $d_\FS(c,c_0) < \epsilon$ and $0 < \tau_c < \ell$ we have $V_c \subseteq V_{c_0}$.
\end{lemma}

Recall that have defined the group $K_c$ to be the $G(K)$-isotropy group of $c$
in~\ref{K_c}.

\begin{proof}[Proof of Lemma~\ref{lem:isotropy-decreases-V}]
  Let $c \in \FS(A)$ with $\tau_c > 0$, i.e., $c$ is periodic.
  Lemma~\ref{lem:about-tau_c} tells us that there is $v \in V_c$ such that
  $\Phi_{\tau_c}(c) = vc$ and that $v$ together with $K_c$ generates $V_c$.  The result
  follows therefore from Lemma~\ref{lem:translations-in-nbhd-redone}
\end{proof}

\begin{proposition}\label{prop:p-adic-groups-have-good-FS_0} The action of $G(K)$ on
  $\FS(X)$ satisfies Assumption~\ref{assum:good-FS_0}.
\end{proposition}

\begin{proof} As discussed above the action of the subgroup $N(K)$ on $A$ is cocompact.
  This implies that the action of $N(K)$ on $\FS(A)$ is cocompact as well, see
  Lemma~\ref{lem:G-on-FS-proper}.  Thus we find $\FS_0 \subseteq \FS(A)$ compact with
  $N(K) \cdot \FS_0 = \FS(A)$.  By Lemma~\ref{lem:geodesics-in-apartments} we have
  $G(K) \cdot \FS(A) = \FS$.  So $G \cdot \FS_0 = \FS$,
  i.e.,~\ref{assum:good-FS_0:fund-domain} is satisfied.  
  
  Towards~\ref{assum:good-FS_0:V}, we first observe that $\tau_c < \infty$ implies $c \in \FS_\infty$, see Lemma~\ref{lem:about-tau_c}.  
  Let now $c_0 \in \FS_0 \subseteq \FS(A)$ and $\ell > 0$ be given.
  We need to find an open neighborhood $U$ of $c_0$ in $\FS_0$ such that for all $c \in U$ with $\tau_c < \ell$ we have $V_c \subseteq V_{c_0}$.
  As $\FS_\infty \subseteq \FS$ is closed we can take $\FS_0 \setminus \FS_\infty$ if $c_0 \not\in \FS_\infty$.
  If $c_0 \in \FS_\infty$, then Lemma~\ref{lem:isotropy-decreases-V} provides a suitable $\epsilon$-neighborhood.
  Thus~\ref{assum:good-FS_0:V} is satisfied as well.
\end{proof}


\typeout{--------- Section: Basics about group actions  ----------}

\section{Basics about group actions}\label{app:basics}


A (continuous) map $f \colon X \to Y$ of (compactly generated topological) spaces is
called \emph{proper} if preimages of compact subsets are compact again. A $G$-space $X$ is
called \emph{proper} if the map
$\Theta^G_X \colon G \times X \to X \times X, \; (g,x) \mapsto (x,gx)$ is proper. It is
called \emph{smooth} if all isotropy group are open. It is called \emph{cocompact} if the
quotient space $X/G$ is compact.

\begin{lemma}\label{lem:properties_of_groups_actions}
  Let $G$ be a locally compact Hausdorff group and let $X$ be a $G$-space.

  \begin{enumerate}
  \item\label{lem:properties_of_groups_actions:proper_in_terms_of_neighborhoods} The
    $G$-space $X$ is proper if and only if for any $x \in X$ there is an open neighborhood
    $U$ such that the subset $\{g \in G \mid g \cdot U \cap U \neq \emptyset \}$ of $G$ is
    relatively compact, i.e., its closure in $G$ is compact;
  
  \item\label{lem:properties_of_groups_actions:proper_implies_compact_isotropy} The
    isotropy groups of a proper $G$-space are all compact.  A $G$-$CW$-complex is proper
    if and only if all its isotropy groups are compact;

  \item\label{lem:properties_of_groups_actions:translating_compact_subgroups} If the
    $G$-space $X$ is proper, then for every compact subset $K \subseteq X$ the subset
    $\{g \in G \mid gK \cap K \not= \emptyset\}$ of $G$ is compact. The converse is true
    if $X$ is locally compact;

  \item\label{lem:properties_of_groups_actions:isotropy_groups_in_a_neighborhood} Let $X$
    be a metric space with isometric proper smooth $G$-action. Then for any $x \in X$
    there exists $\epsilon > 0$ satisfying:

    \begin{enumerate}
    \item\label{lem:properties_of_groups_actions:isotropy_groups_in_a_neighborhood:(1)}
      $G_x =\{g \in G \mid g \cdot B_{\epsilon}(x) \cap B_{\epsilon}(x) \not=
      \emptyset\}$;
    \item\label{lem:properties_of_groups_actions:isotropy_groups_in_a_neighborhood:(2)}
      The map
      \[\alpha \colon G \times_{G_x} B_{\epsilon}(x) \xrightarrow{\cong} G \cdot
        B_{\epsilon}(x), \quad (g,y) \mapsto gy
      \]
      is a $G$-homeomorphism;
    \item\label{lem:properties_of_groups_actions:isotropy_groups_in_a_neighborhood:(3)} We
      have $G_y \subseteq G_x$ for every $y \in B_{\epsilon}(x)$.

    \end{enumerate}
  \item\label{lem:properties_of_groups_actions:proper_and_subspaces} Let $X$ be a proper
    $G$-space. Let $A \subseteq X$ be a closed subspace. Let
    $G_A = \{g \in G \mid gA = A\}$.  Then the $G_A$-space $A$ is proper;

  \item\label{lem:properties_of_groups_actions:proper_and_subspaces:passing_to_subgroups}
    Let $H$ be a (closed) subgroup of $G$. If the $G$-space $X$ is proper, then its
    restriction to an $H$-space is proper.

  \item\label{lem:properties_of_groups_actions:proper_and_subspaces:surjective_maps_and_proper}
    Let $f \colon X \to Y$ be a proper $G$-map. If $Y$ is proper, then $X$ is proper;

  \item\label{lem:properties_of_groups_actions:smooth_and_discrete_orbits} Let $X$ be a
    metric space on which $G$ acts isometrically and properly.  Then $X$ is smooth if and
    only if each orbit $Gx$ (equipped with the subspace topology from $X$) is discrete;

  \item\label{lem:properties_of_groups_actions:proper_and_compact_isotropy} Suppose that
    $X$ is a locally compact metric space on which $G$ acts isometrically.  Then $X$ is
    proper and smooth if and only if each orbit $Gx$ is discrete and each isotropy group
    $G_x$ is compact.

  \item\label{lem:properties_of_groups_actions:cocompact_and_compact_subset} If the
    $G$-space $X$ contains a compact subset $C$ with $G \cdot C = X$, then $X$ is
    cocompact.

    If the $G$-space $X$ is locally compact and cocompact, then there is a compact subset
    $C \subseteq X$ satisfying $G\cdot C = X$;

  \item\label{lem:properties_of_groups_actions:proper_and_subspaces:surjective_maps_and_cocompact}
    Let $f \colon X \to Y$ be a proper $G$-map. If $Y$ is locally compact and cocompact,
    then $X$ is cocompact.

  \end{enumerate}
\end{lemma}
\begin{proof}~\ref{lem:properties_of_groups_actions:proper_in_terms_of_neighborhoods}
  See~\cite[Proposition~3.21 in Chapter~I on page~28]{Dieck(1987)}.
  \\[1mm]~\ref{lem:properties_of_groups_actions:proper_implies_compact_isotropy} Obviously
  the isotropy groups of a proper $G$-space are all compact. The claim about
  $G$-$CW$-complexes is proved in~\cite[Theorem 1.23 on page~18]{Lueck(1989)}.
  \\[1mm]~\ref{lem:properties_of_groups_actions:translating_compact_subgroups} The set
  $\{g \in G \mid g \cdot K \cap K \not= \emptyset\}$ is the image of
  $(\Theta^G_X)^{-1}(K \times K)$ under the projection $G \times X \to G$. Hence it is
  compact if $X$ is proper.  The converse follows from
  assertion~\ref{lem:properties_of_groups_actions:proper_in_terms_of_neighborhoods}.
  \\[1mm]~\ref{lem:properties_of_groups_actions:isotropy_groups_in_a_neighborhood} Suppose
  assertion~\ref{lem:properties_of_groups_actions:isotropy_groups_in_a_neighborhood:(1)}
  is not true. Then there is an $x \in X$ such that for every $\epsilon > 0$ the isotropy
  group $G_x$ is not equal to
  $\{g \in G \mid g \cdot B_{\epsilon}(x) \cap B_{\epsilon}(x) \not= \emptyset$. Hence we
  can choose a sequence of elements $(x_n)_{n \ge 0}$ in $X$ and a sequence of elements
  $(g_n)_{n \ge 0}$ in $G$ such that $x_n$ and $g_nx_n$ belong to $B_{1/n}(x)$ and
  $g_n \notin G_x$ holds.  Because of
  assertion~\ref{lem:properties_of_groups_actions:proper_in_terms_of_neighborhoods} there
  is an open neighborhood $U$ such that the subset
  $\overline{\{g \in G \mid g \cdot U \cap U \neq \emptyset \}}$ of $G$ is compact.  By
  passing to subsequences we can arrange $x_n \in U$ for $n \ge 0$.  Then $g_n$ belongs to
  the compact subset $\overline{\{g \in G \mid g \cdot U \cap U \neq \emptyset \}}$ of
  $G$. By passing to subsequences, we can arrange that there is an element $g \in G$ with
  $\lim_{n \to \infty} g_n = g$.  Since $\lim_{n \to \infty} x_n = x$ and
  $\lim_{n \to \infty} g_nx_n = x$ holds and $G$ acts isometrically, we conclude
  $\lim_{n \to \infty} g_nx = x$. This implies $g \in G_x$ because of
  $\lim_{n \to \infty} g_n = g$.  Since $G_x$ is open in $G$ and
  $\lim_{n \to \infty} g_n = g$, we get for almost all $g_n$ that $g_n = g$ and hence
  $g_n \in G_x$, a contradiction.  This proves
  assertion~\ref{lem:properties_of_groups_actions:isotropy_groups_in_a_neighborhood:(1)}.

  Next we show
  assertion~\ref{lem:properties_of_groups_actions:isotropy_groups_in_a_neighborhood:(2)}
  for the $\epsilon$ for which
  assertion~\ref{lem:properties_of_groups_actions:isotropy_groups_in_a_neighborhood:(1)}
  is true. Let $s \colon G/G_x \to G$ be a map of sets such that its composition with the
  projection $p \colon G \to G/G_x$ the identity on $G/G_x$. Since $G/G_x$ is discrete, we
  obtain a homeomorphism
  \[\beta \colon G/G_x \times B_{\epsilon}(x) \xrightarrow{\cong} G \times_{G_x}
    B_{\epsilon}(x), \quad (gG_x, y) \mapsto (s(gG_x),y).
  \]
  Its inverse is given by $(g,y) \mapsto (gG_x, s(gG_x)^{-1}gy)$. The composite
  $\beta \circ \alpha$ is the map
  $G/G_x \times B_{\epsilon}(x) \to G \cdot B_{\epsilon}(x), (gG_x,y) \mapsto s(gG_x)
  \cdot y$ which is a homeomorphism since $G \cdot B_{\epsilon}(x)$ is the disjoint union
  of open subsets $\coprod_{gG_x \in G/G_x} s(gG_x) \cdot B_{\epsilon}(x)$ by
  assertion~\ref{lem:properties_of_groups_actions:isotropy_groups_in_a_neighborhood:(1)}.
  Hence
  assertion~\ref{lem:properties_of_groups_actions:isotropy_groups_in_a_neighborhood:(2)}
  is true.

  Assertion~\ref{lem:properties_of_groups_actions:isotropy_groups_in_a_neighborhood:(3)}
  follows directly from
  assertion~\ref{lem:properties_of_groups_actions:isotropy_groups_in_a_neighborhood:(1)}.
  \\[1mm]\ref{lem:properties_of_groups_actions:proper_and_subspaces} Let
  $f \colon X \to Y$ be a proper map and $B \subseteq Y$ closed. Consider any closed
  subspace $A \subseteq f^{-1}(B)$.  Then the induced map $f|_{A} \colon A \to B$ is
  obviously proper. Now the claim follows from the commutativity of the following diagram
  \[
    \xymatrix{G_B \times B \ar[r]^-{\Theta^{G_B}_A} \ar[d] & B \times B \ar[d]
      \\
      G \times X \ar[r]_-{\Theta^G_X} & X \times X}
  \]
  whose vertical arrows are the obvious inclusions.
  \\[1mm]~\ref{lem:properties_of_groups_actions:proper_and_subspaces:passing_to_subgroups}
  This follows from the fact that the restriction of a proper map to a closed subspace is
  again proper.
  \\[1mm]\ref{lem:properties_of_groups_actions:proper_and_subspaces:surjective_maps_and_proper}
  Suppose that $Y$ is proper. The following diagram
  \[
    \xymatrix{G \times X \ar[r]^-{\Theta^G_X} \ar[d]_{\id_G \times f} & X \times X
      \ar[d]^{f \times f}
      \\
      G \times Y \ar[r]_-{\Theta^G_Y} & Y \times Y}
  \]
  commutes. Since the lower horizontal arrow and the vertical arrows are proper, the upper
  arrow is proper, see~\cite[Lemma~1.16 on page~14]{Lueck(1989)}, in other words, $X$ is
  proper.  \\[1mm]~\ref{lem:properties_of_groups_actions:smooth_and_discrete_orbits} This
  follows from the fact that for a proper $G$-space the canonical map $G/G_x \to Gx$ is a
  homeomorphism for every $x \in X$, see~\cite[Proposition~3.19~(ii) in Chapter~I on
  page~28]{Dieck(1987)} or~\cite[Lemma~1.19~(iii) on page~16]{Lueck(1989)}.
  \\[1mm]~\ref{lem:properties_of_groups_actions:proper_and_compact_isotropy} Suppose that
  $X$ is proper and smooth.  Then each isotropy group is compact and each $G$-orbit is
  discrete by
  assertions~\ref{lem:properties_of_groups_actions:proper_implies_compact_isotropy}
  and~\ref{lem:properties_of_groups_actions:smooth_and_discrete_orbits}.

  Suppose that each each orbit $Gx$ is discrete and each isotropy group $G_x$ is compact.
  By assertion~\ref{lem:properties_of_groups_actions:smooth_and_discrete_orbits} it
  suffices to show that the $G$-action is proper.  Consider $x \in X$. Since $Gx$ is
  discret and $X$ locally compact, we can choose a compact neighborhood $U$ of $x$ in $X$
  satisfying $U \cap Gx = \{x\}$. It suffices to show that the subset
  $\{g \in G \mid g U \cap U \not= \emptyset\}$ of $G$ is relative compact because of
  assertion~\ref{lem:properties_of_groups_actions:proper_in_terms_of_neighborhoods}.  We
  can equip $G$ with a left invariant proper metric,
  see~\cite{Haagerup-Przybyszewska(2006)}. It suffices to show that any sequence
  $(g_n)_{n \in \IN}$ of element in $G$ satisfying $g_n\cdot U \cap U \not= \emptyset$ has
  a convergent subsequence. Choose for each $n \ge 0$ elements $u_n$ and $u_n'$ in $U$
  such that $g_n u_n = u_n'$ holds. Since $U$ is compact, we can arrange after passing to
  subsequences that $\lim_{n \to \infty} u_n = u$ and $\lim_{n \to \infty} u_n' = u'$
  holds for appropriate elements $u,u' \in U$.  Since $G$ acts by isometries on $X$, we
  get $\lim_{n \to \infty} g_n u = u'$. As $Gu$ is a discrete subspace of $X$, we can
  arrange after passing to subsequences that $g_nu = u'$ holds for $n \ge 0$.  Hence
  $g_0^{-1}g_n u = u$ for all $n \ge 0$. Since $G_u$ is compact, we can arrange after
  passing to subsequences that $g_0^{-1}g_n$ is a convergent sequence in $G$. This implies
  that $g_n$ is a convergent sequence in $G$.
  \\[1mm]~\ref{lem:properties_of_groups_actions:cocompact_and_compact_subset} Let
  $p \colon X \to X/G$ be the projection. If $C \subseteq X$ is a compact subset with
  $G \cdot C = X$, then $p(C) = X/G$ and hence $X/G$ is compact.

  Suppose that $X$ is locally compact and $X/G$ is compact.  We can find for every
  $x \in X$ an open neighborhood $U(x)$ such that $\overline{U(x)}$ is compact.  We obtain
  an open covering $\{p(U(x)) \mid x \in X\}$ of $X/G$.  Since $X/G$ is compact, we can
  find a finite subset $I \subseteq X$ with $\bigcup_{x \in I} p(U(x)) = X/G$.  Then
  $C = \bigcup_{x \in I} \overline{U(x)}$ is a compact subset of $X$ with $G \cdot C = X$.
  \\[1mm]~\ref{lem:properties_of_groups_actions:proper_and_subspaces:surjective_maps_and_cocompact}
  Since $Y$ is locally compact and cocompact, we can choose by
  assertion~\ref{lem:properties_of_groups_actions:cocompact_and_compact_subset} a compact
  subset $C \subseteq Y$ with $G \cdot C = Y$.  Since $f$ is proper, $D:=f^{-1}(C)$ is a
  compact subset of $X$. Since $G \cdot D = X$ holds, $X/G$ is compact by
  assertion~\ref{lem:properties_of_groups_actions:cocompact_and_compact_subset}, in other
  words, $X$ is cocompact. This finishes the proof of
  Lemma~\ref{lem:properties_of_groups_actions}.
\end{proof}

\begin{example}[Non-proper action]\label{rem:non-proper_action}
  The action of $\IZ$ on $S^1$ by rotating through an irrational angle is free, isometric,
  and smooth, but not proper. All $\IZ$-orbits are dense and not discrete. The canonical
  $G$-map map $G/G_x \to Gx$ is continuous and bijective, but not a homeomorphism.
\end{example}


\typeout{--------- Section: References  ----------}

\def\cprime{$'$} \def\polhk#1{\setbox0=\hbox{#1}{\ooalign{\hidewidth
  \lower1.5ex\hbox{`}\hidewidth\crcr\unhbox0}}}


\begin{thebibliography}{10}

\bibitem{Abels-Manoussos-Noskov-proper-inv-metric(2011)}
H.~Abels, A.~Manoussos, and G.~Noskov.
\newblock Proper actions and proper invariant metrics.
\newblock {\em J. Lond. Math. Soc. (2)}, 83(3):619--636, 2011.

\bibitem{Bartels(2018)}
A.~Bartels.
\newblock {$K$}-theory and actions on {E}uclidean retracts.
\newblock In {\em Proceedings of the {I}nternational {C}ongress of
  {M}athematicians---{R}io de {J}aneiro 2018. {V}ol. {II}. {I}nvited lectures},
  pages 1041--1062. World Sci. Publ., Hackensack, NJ, 2018.

\bibitem{Bartels-Lueck(2012CAT(0)flow)}
A.~Bartels and W.~L{\"u}ck.
\newblock Geodesic flow for {C}{A}{T}(0)-groups.
\newblock {\em Geom. Topol.}, 16:1345--1391, 2012.

\bibitem{Bartels-Lueck(2012annals)}
A.~Bartels and W.~L{\"u}ck.
\newblock The {B}orel conjecture for hyperbolic and {CAT(0)}-groups.
\newblock {\em Ann. of Math. (2)}, 175:631--689, 2012.

\bibitem{Bartels-Lueck(2023K-theory_red_p-adic_groups)}
A.~Bartels and W.~L{\"u}ck.
\newblock Algebraic {$K$}-theory of reductive $p$-adic groups.
\newblock Preprint, arXiv:2306.03452 [math.KT], 2023.

\bibitem{Bartels-Lueck-Reich(2008cover)}
A.~Bartels, W.~L{\"u}ck, and H.~Reich.
\newblock Equivariant covers for hyperbolic groups.
\newblock {\em Geom. Topol.}, 12(3):1799--1882, 2008.

\bibitem{Bartels-Lueck-Reich(2008hyper)}
A.~Bartels, W.~L{\"u}ck, and H.~Reich.
\newblock The {$K$}-theoretic {F}arrell-{J}ones conjecture for hyperbolic
  groups.
\newblock {\em Invent. Math.}, 172(1):29--70, 2008.

\bibitem{Baum-Connes-Higson(1994)}
P.~Baum, A.~Connes, and N.~Higson.
\newblock Classifying space for proper actions and ${K}$-theory of group
  ${C}\sp \ast$-algebras.
\newblock In {\em $C\sp \ast$-algebras: 1943--1993 (San Antonio, TX, 1993)},
  pages 240--291. Amer. Math. Soc., Providence, RI, 1994.

\bibitem{Bridson-Haefliger(1999)}
M.~R. Bridson and A.~Haefliger.
\newblock {\em Metric spaces of non-positive curvature}.
\newblock Springer-Verlag, Berlin, 1999.
\newblock Die Grundlehren der mathematischen Wissenschaften, Band 319.

\bibitem{Bruhat-Tits(1972)}
F.~Bruhat and J.~Tits.
\newblock Groupes r\'eductifs sur un corps local.
\newblock {\em Inst. Hautes \'Etudes Sci. Publ. Math.}, 41:5--251, 1972.

\bibitem{Bruhat-Tits(1984)}
F.~Bruhat and J.~Tits.
\newblock Groupes r\'eductifs sur un corps local. {II}. {S}ch\'emas en groupes.
  {E}xistence d'une donn\'ee radicielle valu\'ee.
\newblock {\em Inst. Hautes \'Etudes Sci. Publ. Math.}, 60:197--376, 1984.

\bibitem{Engelking(1978)}
R.~Engelking.
\newblock {\em Dimension theory}.
\newblock North-Holland Publishing Co., Amsterdam-Oxford-New York; PWN---Polish
  Scientific Publishers, Warsaw, 1978.
\newblock Translated from the Polish and revised by the author, North-Holland
  Mathematical Library, 19.

\bibitem{Farrell-Jones(1993a)}
F.~T. Farrell and L.~E. Jones.
\newblock Isomorphism conjectures in algebraic ${K}$-theory.
\newblock {\em J. Amer. Math. Soc.}, 6(2):249--297, 1993.


\bibitem{Haagerup-Przybyszewska(2006)}
U.~Haagerup and A.~Przybyszewska.
\newblock Proper metrics on locally compact groups, and proper affine isometric
  actions on {B}anach spaces.
\newblock Preprint, arXiv:0606794 [math.OA], 2006.

\bibitem{Kasprowski-Rueping(2017long-thin)}
D.~Kasprowski and H.~R\"uping.
\newblock Long and thin covers for flow spaces.
\newblock {\em Groups Geom. Dyn.}, 11(4):1201--1229, 2017.

\bibitem{Kasprowski-Rueping(2017cov)}
D.~Kasprowski and H.~R\"{u}ping.
\newblock Long and thin covers for flow spaces.
\newblock {\em Groups Geom. Dyn.}, 11(4):1201--1229, 2017.

\bibitem{Lueck(1989)}
W.~L{\"u}ck.
\newblock {\em Transformation groups and algebraic ${K}$-theory}, volume 1408
  of {\em Lecture Notes in Mathematics}.
\newblock Springer-Verlag, Berlin, 1989.

\bibitem{Lueck(2005s)}
W.~L{\"u}ck.
\newblock Survey on classifying spaces for families of subgroups.
\newblock In {\em Infinite groups: geometric, combinatorial and dynamical
  aspects}, volume 248 of {\em Progr. Math.}, pages 269--322. Birkh\"auser,
  Basel, 2005.
  
\bibitem{Lueck-Reich(2005)}
W.~L{\"u}ck and H.~Reich.
\newblock The {B}aum-{C}onnes and the {F}arrell-{J}ones conjectures in {$K$}-
  and {$L$}-theory.
\newblock In {\em Handbook of $K$-theory. Vol. 1, 2}, pages 703--842. Springer,
  Berlin, 2005.
  

\bibitem{Schneider-Stuhler(1997)}
P.~Schneider and U.~Stuhler.
\newblock Representation theory and sheaves on the {B}ruhat-{T}its building.
\newblock {\em Inst. Hautes \'Etudes Sci. Publ. Math.}, 85:97--191, 1997.

\bibitem{Tits(1979)}
J.~Tits.
\newblock Reductive groups over local fields.
\newblock In {\em Automorphic forms, representations and $L$-functions (Proc.
  Sympos. Pure Math., Oregon State Univ., Corvallis, Ore., 1977), Part 1},
  Proc. Sympos. Pure Math., XXXIII, pages 29--69. Amer. Math. Soc., Providence,
  R.I., 1979.

\bibitem{Dieck(1987)}
T.~tom Dieck.
\newblock {\em Transformation groups}.
\newblock Walter de Gruyter \& Co., Berlin, 1987.

\bibitem{Wegner(2012)}
C.~Wegner.
\newblock The {$K$}-theoretic {F}arrell-{J}ones conjecture for {CAT}(0)-groups.
\newblock {\em Proc. Amer. Math. Soc.}, 140(3):779--793, 2012.

\end{thebibliography}


\end{document}